\documentclass[11pt,a4paper]{amsart}

\usepackage[margin=35truemm]{geometry}
\usepackage[all]{xy}
\usepackage[utf8]{inputenc}
\usepackage{amsmath,bm}
\usepackage{amssymb}
\usepackage{amsthm}
\usepackage{mathtools}
\usepackage{cases}
\usepackage{float}
\usepackage{here}
\usepackage{comment}
\usepackage[T1]{fontenc}
\usepackage{textcomp}
\usepackage{graphicx, color}
\usepackage{extarrows} 
\usepackage[usenames]{xcolor} 
\usepackage{tikz,epic,eso-pic}

\usepackage{enumitem}
\usepackage[hypertexnames=false,hidelinks]{hyperref} 
\usepackage[capitalize]{cleveref}
\crefname{equation}{}{}
\crefname{enumi}{}{}
\Crefname{lem}{Lemma}{Lemmas}
\Crefname{prop}{Proposition}{Propositions}
\Crefname{thm}{Theorem}{Theorems}
\Crefname{cor}{Corollary}{Corollaries}
\Crefname{rem}{Remark}{Remarks}
\crefname{enumi}{}{}
\setlist[enumerate]{label=(\arabic{enumi})} 

\theoremstyle{plain}
\newtheorem{thm}{Theorem}[section]
\newtheorem{lem}[thm]{Lemma}
\newtheorem{cor}[thm]{Corollary}
\newtheorem{prop}[thm]{Proposition}

\theoremstyle{definition}
\newtheorem{df}[thm]{Definition}

\newtheorem{eg}[thm]{Example}

\newtheorem{conv}[thm]{Convention}

\theoremstyle{remark}
\newtheorem{rem}[thm]{Remark}

\numberwithin{equation}{section}
\allowdisplaybreaks[3]

\newcommand{\R}{\mathbb{R}}
\newcommand{\N}{\mathbb{N}}

\newcommand{\Z}{\mathbb{Z}}
\newcommand{\mr}[2][\Z]{\sigma_{#1}#2}

\newcommand{\Ab}{{\sf Ab}}
\newcommand{\Hom}{{\sf Hom}}
\newcommand{\Mod}{{\sf Mod}}

\newcommand{\Set}{{\sf Set}}

\newcommand{\Met}{{\sf Met}}

\newcommand{\Tor}{{\sf Tor}}

\newcommand{\wt}{\widetilde}

\newcommand{\del}{\partial}

\newcommand{\too}{\longrightarrow}
\newcommand{\lmult}[1]{L_{#1}}
\newcommand{\rmult}[1]{R_{#1}}

\newcommand{\dcpx}{D}
\newcommand{\partialh}{\partial^{\mathrm h}}
\newcommand{\partialv}{\partial^{\mathrm v}}

\renewcommand{\ker}{\operatorname{Ker}}
\DeclareMathOperator{\im}{Im}
\DeclareMathOperator{\tot}{Tot}
\DeclareMathOperator{\MC}{\sf MC}
\DeclareMathOperator{\MResol}{\sf MR}
\DeclareMathOperator{\MH}{\sf MH}
\DeclareMathOperator{\id}{id}

\DeclareMathOperator{\rank}{rank}

\title[Minimal projective resolution and magnitude homology]{%
Minimal projective resolution and magnitude homology of geodetic metric spaces
}

\author[Y. Asao]{Yasuhiko Asao}
\address{%
Faculty of Applied Mathematics, Fukuoka University, Nanakuma, Jonan-ku, Fukuoka, Fukuoka 814-0180, Japan}
\email{asao@fukuoka-u.ac.jp}

\author[S. Wakatsuki]{Shun Wakatsuki}
\address{%
  Graduate School of Mathematics, Nagoya University,
  Furo-cho, Chikusa-ku, Nagoya, Aichi 464-8601, Japan
}
\email{shun.wakatsuki@math.nagoya-u.ac.jp}

\thanks{%
  The first author was supported by JSPS KAKENHI Grant Number 24K16927. The second author was supported by JSPS KAKENHI Grant Number JP23K19006.
}

\begin{document}

\begin{abstract}
In \cite{AI}, magnitude homology is described as a $\Tor$ functor, hence we can compute it by giving a projective resolution of a certain module. In this article, we compute magnitude homology by constructing a minimal projective resolution. As a consequence, we determine magnitude homology of geodetic metric spaces. We show that it is a free $\Z$-module, and give a recursive algorithm for constructing all cycles. As a corollary, we show that a finite geodetic metric space is diagonal if and only if it contains no 4-cuts. Moreover, we give explicit computations for cycle graphs, Petersen graph, Hoffman-Singleton graph, and a missing Moore graph. It includes another approach to the computation for cycle graphs, which has been studied by Hepworth--Willerton (\cite{HW}) and Gu (\cite{Gu}).
\end{abstract}

\maketitle
\setcounter{tocdepth}{1} 
\tableofcontents


\section{Introduction}
{\it Magnitude homology} $\MH^\ell_\ast(X)$ of a (quasi) metric space $X = (X, d)$ is introduced by Hepworth--Willerton (\cite{HW}) and Leinster-Shulman (\cite{LS}) as a categorification of the magnitude that is introduced by Leinster in 2000's.
It is a bigraded module defined as a homology of the {\it magnitude chain complex} $\MC^\ell_\ast(X)$ whose $n$-th component $\MC^\ell_n(X)$ is the free abelian group generated by tuples $(x_0, \dots, x_n) \in X^{n+1}$ with length $\sum_{i=0}^{n-1}d(x_i, x_{i+1}) = \ell$.

In \cite{AI}, the first author and Ivanov showed that magnitude homology can be described as a derived functor.
When $X$ consists of finitely many points, it is  $\Tor_{\mr{X}}(\Z^{|X|}, \Z^{|X|})$, where $\mr{X}$ is a non-commutative graded ring constructed by using the metric space structure of $X$, and the $\Tor$ functor is taken in the abelian category of graded modules over the graded ring $\mr{X}$.
Namely, we can obtain a chain complex $P_\ast\otimes_{\mr{X}} \Z^{|X|}$ that is homotopy equivalent to $\MC^\ast_\ast(X)$ by giving a projective resolution $P_\ast$ of a $\mr{X}$-module $\Z^{|X|}$.

In this article, we compute magnitude homology by using a {\it minimal projective resolution}.
The idea of minimal projective resolution is often used in a field of representation theory.
It is a projective resolution of a module which is
minimal among others with respect to inclusions of complexes.
This can be applied to explicit computation of magnitude homology
since the Jacobson radical of \(\mr{X}\) has a clear description (\cref{rx}).

For a technical reason, we don't need to work with graded modules, so we study things by using usual homological algebra (\cref{gradrem}).  We construct a minimal projective resolution $P_\ast$ of $\Z^{|X|}$ for a finite geodetic quasi metric space $X$ (\cref{mprgeod}). The geodeticity assumption, a discrete analogue of `uniquely geodesic property', has good compatibility with the construction of the resolution. Also, the minimality causes a triviality of all the boundary operators $\del_\ast$ of $P_\ast\otimes_{\mr{X}} \Z^{|X|}$. By describing $P_\ast$ as modules generated by some of tuples $(x_0, \dots, x_n) \in X^{n+1}$ (\cref{thetap}),  we obtain the following. Here we denote the condition $d(x, y) + d(y, z) = d(x, z)$ by $x \leq y \leq z$.

\begin{thm}[\cref{grad}]\label{thm1}
    Let $X$ be a geodetic quasi metric space.
    Then, for all $\ell>0, n>0$, the magnitude homology $\MH_n^\ell(X)$ is a free module whose basis is the set of tuples $(x_0, \dots, x_n) \in X^{n+1}$ satisfying the following.
    \begin{enumerate}
    \item $\sum_{i=0}^{n-1}d(x_i, x_{i+1})=\ell$.
    \item $\lnot(x_{i-1}\leq x_{i} \leq x_{i+1})$ for $1\leq i \leq n-1$.
    \item $x_{0}\leq a \leq  x_{1}$ implies that $a = x_0$ or $a = x_1$.
    \item $x_{i}\leq a \leq  x_{i+1}$ and $a \neq x_{i+1}$ implies that $x_{i-1}\leq x_{i} \leq a$ for $1\leq i \leq n-1$.
    \end{enumerate}
\end{thm}

Here the finiteness condition for $X$, which was necessary to consider $\Tor_{\mr{X}}(\Z^{|X|}, \Z^{|X|})$, is dropped because of the filtered colimit preservation property of magnitude homology (\cref{prescol}).

In the context of previous research on magnitude homology, the above \cref{thm1} carries significant importance as it overcomes difficulties related to  {\it 4-cuts} (See \cref{df4cut} for 4-cuts). In fact, \cref{thm1} can be considered as a development of Kaneta and Yoshinaga's result in \cite{KY} (\cref{KYrem}).
In that paper, they compute magnitude homology $\MH^\ell_\ast(X)$ of a geodetic metric space $X$ in the range $0 < \ell < m_X$, where $m_X$ is the infimum of the length of 4-cuts. They describe magnitude homology by using the notion of {\it thin frame} that is equivalent to the conditions (1)--(4) above under the assumption $0 < \ell < m_X$ (\cref{thintheta}). Also, we can drop the 4-cut-free assumption in Jubin's triviality theorem for Menger convex metric spaces (\cref{menger}).

Next we discuss the {\it diagonality} of metric spaces that is an intriguing property in magnitude theory. From \cref{thm1}, we obtain the following criterion for the  diagonality.
\begin{thm}[\cref{dcritf}]\label{thm2}
    Let $X$ be a finite geodetic quasi metric space. Then the following are equivalent.
    \begin{enumerate}
        \item $X$ is diagonal.
        \item $X$ is 2-diagonal.
        \item There is no 4-cut in $X$.
    \end{enumerate}

\end{thm}
Here we adopt the definition of the diagonality of a quasi metric space due to Bottinelli and Kaiser (\cite{BK}), and {\it $n$-diagonality} is a restriction of the definition of diagonality to $\MH^\ell_n$ for a specific $n\geq0$. Since the diagonality shows simplicity of the magnitude homology, \cref{thm2} can explain why the existence of a 4-cut caused difficulty for computing magnitude homology in \cite{KY}. Interestingly, we cannot drop the finiteness condition in \cref{thm2} (\cref{infremark}).

As another application of \cref{thm1}, we compute magnitude homology of Moore graphs, namely odd cycle graphs, Petersen graph, Hoffman--Singleton graph, and a missing Moore graph (Examples \ref{egodd}--\ref{egmoore}).
Here we consider graphs as metric spaces by the shortest path metric, and note that the above graphs are all geodetic as metric spaces.
It contains the previous studies for odd cycle graphs $C_{2m+1}$--- Hepworth and Willerton conjectured the rank of $\MH^\ell_n(C_{2m+1})$ in \cite{HW}, and Gu proved it by using algebraic Morse theory in \cite{Gu}.
In particular, Gu showed that the rank is determined by a recurrence formula.
Here we show that magnitude homology of Moore graphs are torsion free and  their ranks are determined by similar recurrence formulas generalizing one for odd cycle graph case.
Moreover, we give explicit description of all cycles of them.

\begin{thm}[\cref{compute}, \cref{rem:moore_cycle}]
    Let $G$ be a Moore graph with degree $D$ and diameter $m > 1$. We denote the number of vertices by $N$. Then the magnitude homology $\MH^\ell_n(G)$ is a free $\Z$-module whose rank $R(n, \ell)$ is determined by the recurrence formula
    \[
    R(n, \ell) = R(n-1, \ell-1) + D(D-1)^mR(n-2, \ell - m-1),
    \]
    with the initial condition $R(0, 0) = N, R(1, 1) = ND$.
    Explicitly, we have $\MH^\ell_n(G)=0$ except for the case that \((n,\ell) = (2i+j, (m+1)i+j)\) for some \(i,j\ge 0\). In this case, we have
\begin{equation*}
R(2i+j, (m+1)i+j) =  N\left(D(D-1)^m\right)^i\left({i+j-1 \choose i-1} + D {i+j-1 \choose i}\right),
\end{equation*}
where \(s \choose t\) is the binomial coefficient and
  we define \({-1 \choose -1} = 1\) and \({s \choose t} = 0\) if \(s < t\) or \(t<0\leq s\).
  The basis of $\MH^{(m+1)i+j}_{2i+j}(G)$ are tuples $(x_0, \dots, x_{2i+j})$ described as follows :
\begin{enumerate}
    \item $d(x_0, x_1) = 1, d(x_k, x_{k+1}) \in \{1, m\}$ for $1 \leq k \leq 2i+j-1$ and $d(x_k, x_{k+1}) = m$, and $k$ with $d(x_k, x_{k+1}) = m$ appears exactly $i$ times,
    \item $d(x_{k-1}, x_k) = d(x_k, x_{k+1}) = 1$ implies $x_{k-1} = x_{k+1}$,
    \item $d(x_k, x_{k+1}) = m$ implies $d(x_{k-1}, x_k) = 1 $ and $\neg(x_k \leq x_{k-1}\leq x_{k+1})$,
\end{enumerate}
\end{thm}
Namely the cycles are tuples of an arrange of ``back-and-forth'' and ``distance $m$''. Note that such tuples can be constructed recursively from the above conditions.

Furthermore, we give a computation for magnitude homology of even cycle graphs (\cref{thm:evenref}), whose rank is also conjectured by Hepworth and Willerton, and determined by Gu. As for the Moore graphs, we provide an explicit basis for it, which consists of linear combinations of tuples, instead of single tuples as was the case for the odd scenario. Since even cycle graphs are not geodetic as metric spaces, we cannot apply \cref{thm1}. Instead, we construct suitable projective resolutions for them, which are in fact minimal projective resolutions (\cref{thm:resol_C2m}).

\section{Preliminaries}
\subsection{Conventions for metric spaces}
In the following, we consider a {\it quasi metric space} that is a pair $(X, d)$ of a set $X$ and a map $d : X\times X \too [0, \infty]$ satisfying the following :
\begin{enumerate}
\item $d(x, y) = 0$ if and only if $x = y$,
\item $d(x, y) + d(y, z) \geq d(x, z)$.
\end{enumerate}
The notion of  quasi metric space is a generalization of the classical metric space, namely it admits infinite distance and the distance function may be non-symmetric. The readers who are not familiar with this notion can deal with them as the classical one in the following.
\begin{df}[\cite{LS} Definition 4.19, \cite{KY} Section 2.2]
Let $(X, d)$ be a quasi metric space.
\begin{enumerate}
    \item For $x, y, z \in X$, we denote the condition $d(x, y)+ d(y, z) = d(x, z)$ by $x\leq y\leq z$.
    \item For $x, y, z, w \in X$, we denote the condition $y= z$ and $x\leq y\leq w$ by $x\leq y=z\leq w$.
    \item For $a, b \in X$ with $d(a, b)<\infty$, we denote the set $\{x \in X \mid a \leq  x \leq  b\}$ by $I_X(a, b)$. We can equip $I_X(a, b)$ a poset structure by defining $x \preceq y$ if and only if $a \leq x \leq y$ (or equivalently $x \leq y \leq b$).
    \item $(X, d)$ is called {\it geodetic} if $I_X(a, b)$ is totally ordered for all $a, b \in X$ with $d(a, b)<\infty$.
    \item A quasi metric space $(Y, d_Y)$ is a {\it quasi metric subspace} of $(X, d_X)$ if $Y$ is a subset of $X$ and $d_Y = d_X\mid_Y$.
\end{enumerate}
\end{df}
\begin{rem}
A  (di)graph $G$ is geodetic if and only if there exists at most one shortest (directed) path between any two points. A geodesic space is geodetic if and only if it is uniquely geodesic (\cite{J} Proposition 6.3).
\end{rem}
\begin{lem}\label{subgeod}
    Let $(X, d)$ be a geodetic quasi metric space and $Y \subset X$ be a quasi metric subspace. Then $(Y, d)$ is also geodetic.
\end{lem}
\begin{proof}
    For $a, b \in Y$ with $d(a, b) < \infty$, the poset $I_Y(a, b)$ is a subposet of $I_X(a, b)$, so the statement follows.
\end{proof}
\begin{conv}
\begin{enumerate}
\item We call $(X, d)$ a {\it finite} quasi metric space if $|X|$ is finite.
\item We denote $X^n_f := \{(x_0, \dots, x_{n-1}) \in X^{n}\mid \sum_{i=0}^{n-2}d(x_{i}, x_{i+1})<\infty\}$.
\item We define $X^2_{f+} := \{(x_0, x_1) \in X^2_f \mid x_0\neq x_1\}$.
\item We denote each element $(x, y) \in X^2_f$ by $e_{xy}$. We also denote $e_{xx}$ by $e_x$.
    \end{enumerate}
\end{conv}
In the following, we sometimes denote a quasi metric space $(X, d)$ simply by $X$.
\subsection{Bar resolution}\label{comonad}
In the following, any ring is unital and associative, and any module over a ring is a right module. For a ring $R$, we denote the category of $R$-modules by $\Mod(R)$, which is an abelian category. For a ring homomorphism $f : R \too S$, we denote the induced functor $\Mod(S) \too \Mod(R)$ by $f^\ast$. Note that the functor $f^\ast$ has left and right  adjoints : $-\otimes_R S\dashv f^\ast \dashv \Hom_R(S, -)$. Now we explain a method to construct a projective resolution of a  $S$-module $M$ by using the above functors. Although it is a special case of a more general and well-known framework using comonad (See \cite{Wei} 8.6.12 or \cite{S} Lemma 1.2, for example), we introduce it in a more explicit manner. In the following, we denote a $S$-module $(f^\ast M)\otimes_RS$ by $M\otimes_RS$ for all  $S$-module $M$. We also denote the $n$-times iteration of the functor $(f^\ast -)\otimes_RS$ by $\otimes^n_RS$. We define a $S$-homomorphism $\del_{n, i} : M\otimes^{n+1}_RS \too M\otimes^n_RS$  by
\[
\del_{n, i}(m\otimes s_0\otimes \dots \otimes s_n) = m\otimes s_0 \otimes \dots s_{i-2}\otimes s_{i-1}s_i\otimes s_{i+1} \otimes \dots \otimes s_n,
\]
for $0\leq i \leq n$, where we formally set $s_{-1} = m$.

\begin{lem}\label{chaincpx}
    For a  $S$-module $M$, the homomorphism $\del_{n} := \sum_{i=0}^n(-1)^i\del_{n, i} : M\otimes^{n+1}_RS \too M\otimes^{n}_RS$ define a chain complex
    \[
    \dots \xlongrightarrow{\del_{n+1}} M\otimes^{n+1}_RS \xlongrightarrow{\del_{n}}  M\otimes^{n}_RS  \xlongrightarrow{\del_{n-1}} \dots \xlongrightarrow{\del_{1}} M\otimes_RS \xlongrightarrow{\del_{0}} M \too 0.
    \]
\end{lem}
\begin{proof}
    It is enough to show the simplicial identity $\del_{n-1, i}\del_{n, j} = \del_{n-1, j-1}\del_{n, i}$ for all $0\leq i < j \leq n$, which can be verified immediately.
\end{proof}
\begin{lem}\label{exact}
    The chain complex in Lemma \ref{chaincpx} is exact.
\end{lem}
\begin{proof}
     We define an $R$-homomorphism $h_{n} : M\otimes^{n+1}_RS \too M\otimes^{n+2}_RS$ by $h_{n}(m\otimes s_0\otimes \dots \otimes s_n) = (-1)^{n+1}m\otimes s_0\otimes \dots \otimes s_n \otimes 1$ for $n\geq -1$. Then we can verify that $h_{n-1}\del_n + \del_{n+1}h_n = {\rm id}$, namely $h_\ast$ defines a contraction over $R$. In particular the chain complex is exact.
\end{proof}

\begin{lem}\label{presproj}
The functor $-\otimes_R S$ sends projectives to projectives. Also, if $S$ is projective over $R$, then the functor $f^\ast$ sends projectives to projectives.
\end{lem}
\begin{proof}
    Note that $f^\ast$ preserves surjections, and so does $\Hom_R(S, -)$ when $S$ is projective over $R$. Now the statement follows from the natural isomorphisms $\Hom_R(-, f^\ast=) \cong \Hom_S(-\otimes_RS, =)$ and $\Hom_S(-, \Hom_R(S, =)) \cong \Hom_R(f^\ast -, =)$.
\end{proof}
\begin{cor}\label{barresol}
    If $f^\ast M$ is a projective  $R$-module and $S$ is projective over $R$, then the chain complex in Lemma \ref{chaincpx} is a projective resolution of a  $S$-module $M$.
\end{cor}
\begin{proof}
    It follows from Lemmas \ref{exact} and \ref{presproj}.
\end{proof}

This projective resolution is called the \textit{bar resolution}.

\begin{rem}\label{reduced}
    When the homomorphism $f : R \too S$ admits a retraction $g : S \too R$ with $g \circ f = {\rm id}$, an $R$-module $M$ admits a projective resolution
    \[
    \dots \too (M\otimes^{n}_R\overline{S})\otimes_RS \too  (M\otimes^{n-1}_R\overline{S})\otimes_RS  \too \dots \too M\otimes_RS \too M \too 0,
    \]
    where $\overline{S} = {\rm ker}g$ under the same assumption in \cref{barresol}. It is called the {\it normalized bar resolution} (\cite{Wei} Exercise 8.6.4).
\end{rem}

\subsection{Minimal projective resolution}
As in \cref{comonad}, we suppose that any ring is unital and associative, and any module over a ring is a right module. In the following, we collect some basic facts on the minimal projective resolution from \cite{HGK} and \cite{AF}. We put proofs even for classical results for the readability. We only present facts that is necessary for our purpose, and readers are encouraged to consult the references.
\subsubsection{radical}
\begin{df}
Let $R$ be a ring and $M$ be an $R$-module.
\begin{enumerate}
    \item We define the {\it radical} of $M$ by the intersection of all maximal $R$-submodules of $M$. We denote the radical of $M$ by ${\rm rad}(M)$.
    \item We define the {\it Jacobson radical} of $R$ by the radical of $R$ as an $R$-module.
\end{enumerate}
\end{df}
Note that the kernel of a non-zero homomorphism from a module to a simple module is a maximal submodule. Conversely, any quotient by a maximal submodule is a simple module. Hence we can identify ${\rm rad}(M)$ with the elements $m \in M$ which vanishes by any $R$-homomorphism from $M$ to a simple module $U$.
\begin{lem}[\cite{HGK} Proposition 3.4.3]\label{radcoprod}
    We have ${\rm rad}(\oplus_jM_j) = \oplus_j{\rm rad}(M_j)$ as submodules of $\oplus_j M_j$.
\end{lem}
\begin{proof}
Let $m \in {\rm rad}(\oplus_jM_j)$ and $\pi_i : \oplus_jM_j \too M_i$ be the projection. Then $f_i(\pi_i m) = f_i\pi_i(m) = 0$ for all $f_i : M_i \too U$ with $U$ being simple. Hence we obtain $\pi_i m \in {\rm rad}(M_i)$ and $m \in \oplus_j{\rm rad}(M_j)$, namely ${\rm rad}(\oplus_jM_j) \subset \oplus_j{\rm rad}(M_j)$. Conversely, let $m_i \in {\rm rad}(M_i)$ and $\iota_i : M_i \too \oplus_jM_j$ be the inclusion.  Then $f(\iota_im_i) = f\iota_i(m_i) = 0$ for all  $f : M \too U$ with $U$ being simple. Hence we obtain $\iota_i m_i \in {\rm rad}(M)$, namely $\oplus_j{\rm rad}(M_j) \subset {\rm rad}(\oplus_jM_j)$.
\end{proof}
\begin{lem}\label{radquot}
For an $R$-module $M$ and an $R$-submodule $N \subset {\rm rad}(M) \subset M$, we have ${\rm rad}(M/N) = ({\rm rad}M)/N$ as submodules of $M/N$.
\end{lem}
\begin{proof}
Note that the image and the inverse image of the quotient map $M \too M/N$ induces a bijection between sets of submodules of $M/N$ and submodules of $M$ containing $N$, that preserves inclusions and intersections. Since ${\rm rad}(M/N)$ is the intersection of  maximal submodules of $M/N$, we obtain the statement by the assumption $N \subset {\rm rad}M$.
\end{proof}
\begin{lem}[\cite{HGK} Proposition 5.1.8]\label{projrad}
    For a projective $R$-module $P$, we have ${\rm rad}(P) = P{\rm rad}(R)$.
\end{lem}
\begin{proof}
    Since $R{\rm rad}(R) = {\rm rad}(R)$, we have $F{\rm rad}(R) = {\rm rad}(F)$ for all free $R$-module $F$ by \cref{radcoprod}. When $P$ is projective, there is an $R$-module $Q$ such that $P\oplus Q$ is a free module. Hence we have $P{\rm rad}(R)\oplus Q{\rm rad}(R) = (P\oplus Q){\rm rad}(R) =  {\rm rad}(P\oplus Q) = {\rm rad}(P)\oplus{\rm rad}(Q)$, which preserves inclusions into $P\oplus Q$. Thus it implies that ${\rm rad}(P) = P{\rm rad}(R)$.
\end{proof}
\begin{lem}[\cite{HGK} Proposition 3.4.5]\label{invertible}
    For a ring $R$, we have ${\rm rad}(R) = \{r \in R \mid 1-rs \text{ is right invertible for all \ } s \in R\}$.
\end{lem}
\begin{proof}
    Let $r \in {\rm rad}(R)$ and $s \in R$. If $1-rs$ is not right invertible, then there is a maximal ideal $I$ containing $1-rs$ and $rs$, which is a contradiction. Hence $1-rs$ is  right invertible. Let $r \in R$ such that $1-rs$ is right invertible for all $s \in R$. If there is a maximal ideal $I$ that does not contain $r$, then we have $rR + I = R$. Hence there exist $s\in R, i \in I$ such that $rs + i = 1$, which implies a contradiction that $i \in I$ is right invertible. Hence $r \in {\rm rad}(R)$.
\end{proof}

\subsubsection{small ideal}
\begin{df}
    Let $M$ be an $R$-module and $S \subset M$ be an $R$ submodule. We say that $S$ is {\it small} if $S+N = M$ implies $N=M$ for all $R$ submodule $N \subset M$.
\end{df}
\begin{lem}[Nakayama's lemma, \cite{HGK} Lemma 3.4.12]\label{nakayama}
For a finitely generated $R$-module $M$, the submodule $M{\rm rad}(R) \subset M$ is small.
\end{lem}
\begin{proof}
    Suppose that $M = \sum_{i=1}^km_iR$. We first show that $M{\rm rad}(R) = M$ implies $M=0$. By $M{\rm rad}(R) = M$, there exist $r_1, \dots, r_k \in {\rm rad}(R)$ such that $\sum_{i=1}^km_ir_i = m_k$, which implies that $m_k = (\sum_{i=1}^{k-1}m_ir_i)(1-r_k)^{-1}$ by \cref{invertible}. Hence $M$ is generated by $m_1, \dots, m_{k-1}$, and we obtain $M=0$ by repeating this argument. Let $N \subset M$ be a submodule such that $M{\rm rad}(R) + N = M$. Then we have $(M/N){\rm rad}(R) = (M{\rm rad}(R))/(N\cap{\rm rad}(R)) \cong (M{\rm rad}(R)+N)/N = M/N$, hence $M/N=0$ by the previous argument. Thus we obtain $M=N$.
\end{proof}
\begin{lem}[\cite{AF} Proposition 9.13]\label{smallrad}
    For an $R$-module $M$ and a small submodule $S \subset M$, we have $S \subset {\rm rad}(M)$.
\end{lem}
\begin{proof}
    Let $S \subset M$ be a small submodule and $N \subset M$ be a maximal submodule. Then we have $S \subset N$ by the smallness and the maximality. Hence we obtain $S \subset {\rm rad}(M)$.
\end{proof}
\begin{df}
    A homomorphism of $R$-modules $f : M\too N$ is an {\it essential epimorphism} if it is surjective and ${\rm ker}f \subset M$ is small. Equivalently, $f$ is an essential epimorphism if, it is surjective and for all $g : L \too M$, the surjectivity of $fg$ implies the surjectivity of $g$.
\end{df}

\begin{lem}\label{2-3}
    For surjective homomorpshisms $f : M\too N$ and $g : L \too M$, $f, g, fg$ are essential epimorphisms if two of them are so.
\end{lem}
\begin{proof}
    Straightforward.
\end{proof}
\subsubsection{minimal projective resolution}
\begin{df}
\begin{enumerate}
\item A homomorphism $P \too M$ is called a {\it projective cover} of $M$ if it is an essential epimorphism and $P$ is projective.
\item A projective resolution $\varphi_\ast : P_\ast \too M$ is a {\it minimal projective resolution} if each $\varphi_0 : P_0 \too M$ and $\varphi_n : P_n\too {\rm ker}\varphi_{n-1}$ are projective covers.
\end{enumerate}
\end{df}

\begin{prop}\label{zerodiff}
    Let $\varphi_\ast : P_\ast \too M$ be a minimal projective resolution of an $R$-module $M$. Then $\varphi_n\otimes {\rm id}= 0 : P_n\otimes_R(R/{\rm rad}(R)) \too P_{n-1}\otimes_R(R/{\rm rad}(R))$ for all $n \geq 1$.
\end{prop}
\begin{proof}
    Note that the homomorphism $\varphi_n\otimes {\rm id}$ is isomorphic to the homomorphism $P_n/P_n{\rm rad}(R) \too P_{n-1}/P_{n-1}{\rm rad}(R)$ induced from $\varphi_n$. We have ${\rm im}\varphi_n = {\rm ker}\varphi_{n-1} \subset {\rm rad}(P_{n-1}) = P_{n-1}{\rm rad}(R)$ by \cref{smallrad,projrad}, hence $\varphi_n\otimes {\rm id} = 0$.
\end{proof}
\section{Magnitude homology as a derived functor}\label{mhdf}
In this section, we briefly recall Asao-Ivanov's description (\cite{AI}) of magnitude homology as a derived functor by using explicit descriptions in \cref{comonad}. Let $(X, d)$ be a finite quasi metric space with $|X| = N$.
\begin{df}
We define a free abelian group $\mr{X} := \Z X^2_f$. We equip $\mr{X}$ with a product structure by $e_{xy}\cdot e_{zw} = \begin{cases}e_{xw} & x\leq y=z\leq w \\ 0 & \text{otherwise}\end{cases}$, which makes it an associative ring with a unit $\sum_{x \in X}e_{xx}$. Note that it is necessary that $X$ is a finite set to define the unit. We denote $e_{xx}$ by $e_x$ in the following.
\end{df}

Let $rX := \Z X^2_{f+}$ be an ideal of $\mr{X}$. Consider the quotient $ \mr{X}/rX \cong \bigoplus_{x \in X}\Z\langle e_x\rangle \cong \Z^{N}$ as a ``trivial'' $\mr{X}$-module. Moreover, we can consider this module $\Z^{N}$ as a subring of $\mr{X}$. Now we consider a ring inclusion $\Z^{N} \too \mr{X}$ and apply the framework of Section \ref{comonad}. Note that the $\mr{X}$-module $\Z^N$ is projective as a $\Z^N$-module, and $\mr{X}$ is a free $\Z^N$-module with a basis $\{\sum_{e_{yz} \in X^2_f}e_{x}\cdot e_{yz}\}_{x \in X}$. Hence we obtain a projective resolution of $\mr{X}$-module $\Z^N$ by \cref{barresol}:

\[
    \dots \xlongrightarrow{\del_{n+1}} \Z^N\otimes^{n+1}_{\Z^N}\mr{X} \xlongrightarrow{\del_{n}}  \Z^N\otimes^{n}_{\Z^N}\mr{X}  \xlongrightarrow{\del_{n-1}} \dots \xlongrightarrow{\del_{1}} \Z^N\otimes_{\Z^N}\mr{X} \xlongrightarrow{\del_{0}} \Z^N \too 0.
\]

Here we can describe the $\mr{X}$-module $ \Z^N\otimes^{n+1}_{\Z^N}\mr{X} $ as
\[
 \Z^N\otimes^{n+1}_{\Z^N}\mr{X} \cong \Z X^{n+2}_f,
\]
where the $\mr{X}$-action is described as $(x_0, \dots, x_{n+1})e_{zw} = \begin{cases} (x_0, \dots, x_{n}, w) & x_n\leq x_{n+1} = z \leq w, \\ 0 & \text{otherwise}.\end{cases}$ Further, the homomorphism $\del_n$ is described as $\del_n=\sum_{i=0}^n(-1)^i\del_{n, i}$ with
\begin{align*}
\del_{n, i}(x_0, \dots, x_{n+1}) &= \begin{cases} (x_0, \dots, \hat{x}_i, \dots, x_{n+1}) & x_{i-1}\leq x_i \leq x_{i+1}, \\ 0 & \text{otherwise}, \end{cases}
\end{align*}
for $0\leq i \leq n$, where we formally set $x_{-1} = x_1$ and $\hat{x}$ means deletion of the term $x$. Note here that $x_{-1}\leq x_0 \leq x_1$ implies $x_0 = x_1$. (The obtained projective resolution will be denoted by $\MResol_\ast(X)$ in \cref{comparison}.)

By applying $-\otimes_{\mr{X}}\Z^N$ to this chain complex, we obtain a chain complex $(\wt{\MC}_\ast(X), \wt{\del}_\ast = \sum_{i=0}^\ast (-1)^i\wt{\del}_{n, i})$ where
\begin{align*}
\wt{\MC}_n(X) &= \Z X^{n+1}_f \\
\wt{\del}_{n, i} (x_0, \dots, x_{n}) &= \begin{cases} (x_0, \dots, \hat{x}_i, \dots, x_{n}) & x_{i-1}\leq x_i \leq x_{i+1}, \\ 0 & \text{otherwise}, \end{cases}
\end{align*}
for $0\leq i \leq n$, where we formally set $x_{-1} = x_1$ and $x_{n+1} = x_{n-1}$. This chain complex is exactly the {\it unnormalized magnitude chain complex of} $X$ defined in \cite{LS} Definition 5.7. Its homology is called {\it magnitude homology of} $X$. Namely we obtain the following definition of the magnitude homology as a derived functor.

\begin{df}[\cite{AI}]
For a finite quasi metric space $X$, we define the magnitude homology by $\MH_n(X) = \Tor_n^{\mr{X}}(\Z^N, \Z^N)$.
\end{df}

\begin{rem}
    \label{rem:reducedMC}
    In \cite{LS}, they also defined the {\it normalized magnitude chain complex} $(\MC_\ast(X), \del_\ast)$ which is obtained from a projective resolution of $\Z^N$ in \cref{reduced} with $\overline{S} =  \Z X^2_{f+}$.
\end{rem}
\section{Minimal projective resolution for finite geodetic quasi metric spaces}\label{mprgeodsec}
In this section, we construct a minimal projective resolution of the trivial module $\Z^N$ in \cref{mhdf} for a finite geodetic quasi metric space $(X, d)$ with $|X| = N$. 
\subsection{Definition of $\kappa$ and $\beta$}
Let $X$ be a finite quasi metric space.
First we describe the kernel and image of multiplication maps.
\begin{df}
\begin{enumerate}
\item For all subset  $S \subset X^2_f$, we define a sub abelian group $\bigoplus_{e \in S}\Z\langle e\rangle$ of $\mr{X}$ and denote it by $\Z S$.
\item For all $e_{xy} \in X^2_f$ and $S \subset X^2_f$, we define
\begin{align*}
\kappa(S, e_{xy}) &= \{e_{zw} \in S \mid \neg (z\leq w=x\leq y)\}, \\
\iota(S, e_{xy}) &= \{e_{zy} \in X^2_f \mid e_{zx} \in S, z\leq x \leq y\}, \\
\kappa(e_{xy}, S) &= \{e_{zw} \in S \mid \neg (x\leq y=z\leq w)\}, \\
\iota(e_{xy}, S) &= \{e_{xw} \in X^2_f \mid e_{yw} \in S, x\leq y\leq w\}.\\
\end{align*}
\end{enumerate}
\end{df}
The following is obvious.
\begin{lem}\label{kapiota}
For all $e_{xy} \in X^2_f$ and $S \subset X^2_f$, we have
\begin{align*}
{\rm ker}(\rmult{xy}|_{\Z S}) &= \Z\kappa(S, e_{xy}), \\
{\rm im}(\rmult{xy}|_{\Z S}) &= \Z\iota(S, e_{xy}), \\
{\rm ker}(\lmult{xy}|_{\Z S}) &= \Z\kappa(e_{xy}, S), \\
{\rm im}(\lmult{xy}|_{\Z S}) &= \Z\iota(e_{xy}, S),\\
\end{align*}
where $\lmult{xy} = e_{xy}\cdot(-)$ and $\rmult{xy} = (-)\cdot e_{xy}$ are multiplication homomorphisms $\mr{X} \too \mr{X}$.
\end{lem}

The following definition and lemmas will be used  in our construction of a minimal projective resolution and the other computations.
\begin{df}
\begin{enumerate}
\item For a subset $S \subset X^2_f$, we define $\beta(S) = S\setminus (\bigcup_{e_{xy} \in X^2_{f+}} \iota(S, e_{xy}))$.

\item For all $e_{xy} \in X^2_f$, we define that $\kappa(e_{xy}) := \kappa(e_{xy}, (\{y\}\times X)\cap X^2_f)$. We also define that $\kappa(S) = \bigsqcup_{e_{xy}\in S}\kappa(e_{xy})$ for $S \subset X^2_f$.

\item We define that $\beta (\bigsqcup_i S_i) = \bigsqcup_i \beta(S_i)$ and $\kappa (\bigsqcup_i S_i) = \bigsqcup_i \kappa(S_i)$ for all family $\{S_i \subset X^2_f\}_i$, where $\bigsqcup$ is an external disjoint union.

\end{enumerate}
\end{df}

\begin{lem}\label{kappamod}
    For all $S \subset X^2_f$, the subset $\Z\kappa(S) \subset \bigoplus_{s \in S} \mr{X}$ is a right $\mr{X}$-submodule.
\end{lem}
\begin{proof}
    It follows from the associativity of the product on $\mr{X}$.
\end{proof}

\begin{lem}\label{betakappan}
    For $n\geq 1$ and $S \subset X^2_f$, we have
    \[
    (\beta\kappa)^n(S) = \bigsqcup_{(x_0, x_1)\in S}\bigsqcup_{(x_1, x_2) \in \beta\kappa(e_{x_0x_1})}\dots \bigsqcup_{(x_{n-1}, x_n) \in \beta\kappa(e_{x_{n-2}x_{n-1}})}\beta\kappa(e_{x_{n-1}x_n}).
    \]

\end{lem}

\begin{proof}
    We prove it by induction. For $n=1$, it is obvious that $\beta\kappa(S) = \beta(\bigsqcup_{(x, y) \in S}\kappa(e_{xy})) = \bigsqcup_{(x, y) \in S}\beta\kappa(e_{xy})$  from the definition. Assume that it is true up to $n$. Note that we have $\kappa(\bigsqcup_iS_i) = \bigsqcup_i\bigsqcup_{s_i \in S_i}\kappa(s_i)$ for all family $\{S_i \subset X^2_f\}_i$. Then we have
    \begin{align*}
        &(\beta\kappa)^{n+1}(S)\\
        &= \beta\kappa \left(\bigsqcup_{(x_0, x_1)\in S}\dots \bigsqcup_{(x_{n-1}, x_n) \in \beta\kappa(e_{x_{n-2}x_{n-1}})}\beta\kappa(e_{x_{n-1}x_n}) \right) \\
         &= \beta \left(\bigsqcup_{(x_0, x_1)\in S}\dots \bigsqcup_{(x_{n-1}, x_n) \in \beta\kappa(e_{x_{n-2}x_{n-1}})}\bigsqcup_{(x_n, x_{n+1})\in \beta\kappa(e_{x_{n-1}x_n})}\kappa(e_{x_nx_{n+1}}) \right) \\
         &= \bigsqcup_{(x_0, x_1)\in S}\dots \bigsqcup_{(x_{n-1}, x_n) \in \beta\kappa(e_{x_{n-2}x_{n-1}})}\bigsqcup_{(x_n, x_{n+1})\in \beta\kappa(e_{x_{n-1}x_n})} \beta\kappa(e_{x_nx_{n+1}}).
    \end{align*}
\end{proof}

\begin{lem}\label{geodetic}
For a quasi metric space $(X, d)$, the following are equivalent.
\begin{enumerate}
    \item $(X, d)$ is geodetic,
    \item For all $S \subset X^2_f$ and for all $e_{xy}\neq e_{zw} \in \beta(S)$, we have $\iota(e_{xy}, X^2_f)\cap \iota(e_{zw}, X^2_f) = \emptyset$.
 \end{enumerate}
\end{lem}
\begin{proof}
   $(1)\Rightarrow (2)$: Let $e_{xy}\neq e_{zw} \in \beta(S)$ and $e_{ab} \in \iota(e_{xy}, X^2_f)\cap \iota(e_{zw}, X^2_f)$. Namely we have $x = z = a, y\neq w$ and $x\leq y \leq b, z \leq w \leq b$. Since $X$ is geodetic and $d(a, b)<\infty$, we have $z \leq y \leq w$ or $x \leq w \leq y$, that is $e_{zw} = e_{zy}e_{yw}$ or $e_{xy} = e_{xw}e_{wy}$ respectively. In both cases, by $y\neq w$,  we obtain that $e_{zw} \not\in \beta(S)$ or $e_{xy} \not\in \beta(S)$ that are contradictions. Hence we obtain that  $\iota(e_{xy}, X^2_f)\cap \iota(e_{zw}, X^2_f) = \emptyset$.

   $(2)\Rightarrow (1)$: Let $d(a, b) < \infty$ and $x \neq y \in I(a, b)$, namely we have $a\leq x\neq y \leq b$. Let $S = \{e_{ax}, e_{ay}\}$ and suppose that $\beta(S) =  \{e_{ax}, e_{ay}\}$. Then we have $\iota(e_{ax}, X^2_f)\cap \iota(e_{ay}, X^2_f) = \emptyset$, which contradicts that $e_{ax}e_{xb} = e_{ay}e_{yb} = e_{ab}$. Hence we obtain that $\beta(S)\neq S$ which is equivalent to that $a\leq x \leq y$ or $a \leq y \leq x$, namely $I_X(a, b)$ is totally ordered.
\end{proof}

\subsection{Construction of the resolution}

Recall that $rX := \Z X^2_{f+}$ is an ideal of $\mr{X}$. Also recall that any module and ideal are right ones.
\begin{lem}\label{rx}
The ideal $rX$ is the Jacobson radical of $\mr{X}$.
\end{lem}
\begin{proof}
Since $X$ is a finite set, the function $d$ is bounded on  $X^2_f$  and hence $rX$ is a nilpotent ideal, namely we have $(rX)^k = 0$ for some $k$. Then for all $r \in rX$ and $s \in \mr{X}$,  $(1-rs)$ is invertible since we have $(1-rs)\sum_{i=0}^k(rs)^i = 1$. Hence we obtain that $rX \subset {\rm rad}(\mr{X})$ by \cref{invertible}.  Now we have ${\rm rad}(\mr{X})/ rX\cong {\rm rad}(\mr{X}/ rX) \cong {\rm rad}\Z^n \cong ({\rm rad}\Z)^n \cong 0$ by \cref{radcoprod,radquot}, hence $rX = {\rm rad}(\mr{X})$.
\end{proof}
\begin{lem}\label{emr}
$e_x \in \mr{X}$ is an idempotent for all $x \in  X$. Hence ${\rm im}(\lmult{x}) = e_x\mr{X}$ is a projective  $\mr{X}$-module.
\end{lem}
\begin{proof}
    It is obvious that $e_xe_x = e_x$. Hence the inclusion $e_x\mr{X} \too \mr{X}$ has a retraction $\lmult{x} : \mr{X} \too e_x\mr{X}$, which implies that $e_x\mr{X}$ is a direct summand of $\mr{X}$.
\end{proof}

\begin{prop}\label{mpr1}
Let $X$ be a finite geodetic quasi metric space and $S \subset X^2_f$. If $\Z S$ is a  submodule of $\mr{X}$, then it has a projective cover $\varphi : P \too \Z S$ as a $\mr{X}$-module such that
\begin{enumerate}
\item $P = \bigoplus_{e_{xy} \in \beta(S)}e_{y}\mr{X}$,
\item $\varphi = \bigoplus_{e_{xy} \in \beta(S)}(\lmult{xy})|_{e_{y}\mr{X}}$,
\item ${\rm ker}\varphi = \Z\kappa\beta(S)$.
\end{enumerate}

Consequently, $\Z S$ admits a minimal projective resolution $P_\ast$ with boundary operators $\varphi_i : P_i \too P_{i-1}$ as a $\mr{X}$-module such that
\begin{enumerate}
\item $P_n = \bigoplus_{e_{xy} \in (\beta\kappa)^n\beta(S)}e_{y}\mr{X}$,
\item $\varphi_n = \bigoplus_{e_{xy} \in (\beta\kappa)^n\beta(S)}(\lmult{xy})|_{e_{y}\mr{X}}$,
\item ${\rm ker}\varphi_n = \Z(\kappa\beta)^{n+1}S$,
\end{enumerate}
here we note that the set $(\beta\kappa)^n\beta(S)$ is a subset of $\bigsqcup_{i\in I}X^2_f$ for some index set $I$.
\end{prop}
\begin{proof}
First note that the $\mr{X}$-module $P = \bigoplus_{e_{xy} \in \beta(S)}e_{y}\mr{X}$ is projective by \cref{emr}, and that $\Z\kappa\beta(S)$ is a $\mr{X}$-module by \cref{kappamod}.  We also note that $(\Z S)(rX) = \sum_{e_{xy} \in X^2_{f+}}\Z\iota(S, e_{xy}) = \Z(\bigcup_{e_{xy} \in X^2_{f+}}\iota(S, e_{xy}))$. Hence we have $\Z S/(\Z S)(rX) \cong \Z \beta(S)$. Now we consider the following diagram :
\begin{equation*}
\xymatrix{
\bigoplus_{e_{xy} \in \beta(S)}e_{y}\mr{X} \ar[d]_{\pi'} \ar@{-->}[r]^{\varphi} & \Z S \ar[d]^{\pi} \\
\Z \beta(S) \ar[r]^{\cong} & \Z S/(\Z S)(rX),
}
\end{equation*}
where the left vertical arrow $\pi'$ is a direct sum of  quotients $e_{y}\mr{X} \too e_{y}\mr{X}/(e_{y}\mr{X})(rX)\cong \Z$, and the dotted arrow $\varphi$ is induced from the universality of the projective module $\bigoplus_{e_{xy} \in \beta(S)}e_{y}\mr{X}$. Since $S$ is a finite set, the homomorphism $\pi$ is an essential epimorphism by \cref{nakayama,rx}. Similarly, the homomorphism $\pi'$ is an essential epimorphism, and so is $\varphi$ by \cref{2-3}. Hence $\varphi$ is a projective cover. We can suppose that the homomorphism $\varphi$ is the sum of the left multiplications $ \lmult{xy} : e_{y}\mr{X} \too \Z S$ for all $e_{xy} \in \beta(S)$. By the geodeticity assumption, Lemmas \ref{kapiota} and \ref{geodetic}, we have ${\rm im}\ \lmult{xy} \cap {\rm im}\ \lmult{zw} = \Z(\iota(e_{xy}, X^2_f)\cap \iota(e_{zw}, X^2_f)) = 0$ for all $e_{xy}\neq e_{zw} \in \beta(S)$. Therefore, we have ${\rm ker} \varphi = \bigoplus_{e_{xy} \in \beta(S)}{\rm ker}(\lmult{xy}|_{e_{y}\mr{X}}) = \bigoplus_{e_{xy} \in \beta(S)}\Z \kappa(e_{xy}, (\{y\}\times X)\cap X^2_f) = \Z \kappa(\beta(S))$.
\end{proof}

Now we are ready to construct a minimal projective resolution
of the trivial \(\mr{X}\)-module.
\begin{prop}\label{mprgeod}
Let $X$ be a finite geodetic quasi metric space. Then the trivial $\mr{X}$-module $\mr{X}/rX \cong \Z^N$ admits a minimal projective resolution $P_\ast \too \Z^N$ with boundary operators $\varphi_n : P_n \too P_{n-1}$ as a $\mr{X}$-module such that
\begin{enumerate}
\item $P_0 = \mr{X}$ and $P_n = \bigoplus_{e_{xy} \in (\beta\kappa)^{n-1}\beta(X^2_{f+})}e_{y}\mr{X}$ for $n\geq 1$,
\item  $\varphi_0$ is the quotient map $\mr{X} \too \mr{X}/ rX$ and $\varphi_n = \bigoplus_{e_{xy} \in (\beta\kappa)^{n-1}\beta(S)}(\lmult{xy}|_{e_{y}\mr{X}})$ for $n\geq 1$,
\item ${\rm ker}\varphi_n = \Z(\kappa\beta)^{n}(X^2_{f+})$ for $n\geq 0$.
\end{enumerate}
\end{prop}
\begin{proof}
Since the quotient map $\mr{X} \too \mr{X}/ rX$ is a projective cover and its kernel is $rX$, we can apply Proposition \ref{mpr1} to $S = X^2_{f+}$, which implies the statement.
\end{proof}

\begin{rem}
    Note that only the existence of a minimal projective resolution over $\mr[k]{X}$ is guaranteed for all field $k$ (\cite{HGK} Theorem 10.4.8). Here, we obtained an explicit construction of a minimal resolution over $\mr{X}$.
\end{rem}




\section{Magnitude homology of geodetic quasi metric spaces}
\label{homology_geodetic}
In this section, we describe the projective resolution obtained in \cref{mprgeodsec} in a geometric manner, namely by using the notion \(\Theta_n(S)\). Consequently, together with a filtered colimit preservation property of magnitude homology, we obtain a complete description of magnitude homology of geodetic quasi metric spaces. In the following, $(X, d)$ is a quasi metric space (not necessarily finite or geodetic).
\subsection{Definition of $\Theta_n(S)$}

\begin{df}\label{dftheta}
    Let $S \subset X^2_f$. We define a set $\Theta_n(S)$ for $n\geq 1$ that consists of tuples $(x_0, \dots, x_{n}) \in X^{n+1}$ satisfying the following for all $1\leq i \leq n-1$.
    \begin{enumerate}
    \item $(x_0, x_1) \in S$, $(x_i, x_{i+1}) \in X^2_f$.
    \item $\lnot(x_{i-1}\leq x_{i} \leq x_{i+1})$.
    \item $x_{i}\leq a \leq  x_{i+1}$ and $a \neq x_{i+1}$ imply that $x_{i-1}\leq x_{i} \leq a$.
    \end{enumerate}
\end{df}

\begin{lem}\label{theta1}
$\Theta_{n+1}(S) = \bigsqcup_{(x_0, \dots, x_n) \in \Theta_n(S)}\beta\kappa(e_{x_{n-1}x_{n}})$.
\end{lem}
\begin{proof}
     Note that $\Theta_{n+1}(S)\subset \Theta_n(S)\times X \cong \bigsqcup_{(x_0, \dots, x_n) \in \Theta_n(S)}\{x_n\}\times X$.  Hence we can regard $\Theta_{n+1}(S)$ as a subset of $\bigsqcup_{(x_0, \dots, x_n) \in \Theta_n(S)}\{x_n\}\times X$. By the condition $(2)$ of \cref{dftheta}, we have $\Theta_{n+1}(S) \subset \bigsqcup_{(x_0, \dots, x_n) \in \Theta_n(S)}\kappa(e_{x_{n-1}x_n})$. Suppose $(x_0, \dots, x_{n+1}) \in \Theta_{n+1}(S)$. If $(x_n, x_{n+1}) \not\in \beta\kappa(e_{x_{n-1}x_{n}})$, then there exist $a\neq x_{n+1} \in X$ such that $(x_n, a) \in \kappa(e_{x_{n-1}x_{n}})$ and $x_n \leq a \leq x_{n+1}$. Then the condition (3) of \cref{dftheta} implies that $x_{n-1}\leq x_{n} \leq a$, which implies a contradiction $(x_n, a) \not\in \kappa(e_{x_{n-1}x_{n}})$. Hence we obtain $\Theta_{n+1}(S) \subset \bigsqcup_{(x_0, \dots, x_n) \in \Theta_n(S)}\beta\kappa(e_{x_{n-1}x_n})$. Conversely, we can verify that any element $(x_0, \dots, x_{n+1}) \in \bigsqcup_{(x_0, \dots, x_n) \in \Theta_n(S)}\beta\kappa(e_{x_{n-1}x_{n}})$ belongs to $\Theta_{n+1}(S)$.
\end{proof}

Note that, by \cref{betakappan}, any element of $(\beta\kappa)^nS$ can be expressed as $(x_0, \dots, x_{n+1})$ satisfying that $(x_0, x_1) \in S$ and $(x_{i}, x_{i+1}) \in \beta\kappa(e_{x_{i-1}x_i})$ for $1\leq i \leq n$ for $n \geq 1$.

\begin{lem}\label{theta3}
    By the above expression, we can identify $(\beta\kappa)^nS$ and $\Theta_{n+1}(S)$ for $n \geq 1$.
\end{lem}
\begin{proof}
    We prove it by induction. For $n = 1$, it follows from $\Theta_1(S) = S$ and \cref{theta1}. Assume that it is true up to $n$. Then we have
    \begin{align*}
        \Theta_{n+2}(S) &= \bigsqcup_{(x_0, \dots, x_{n+1}) \in \Theta_{n+1}(S)}\beta\kappa(e_{x_{n}x_{n+1}}) \\
        &= \bigsqcup_{(x_0, \dots, x_{n+1}) \in (\beta\kappa)^nS}\beta\kappa(e_{x_{n}x_{n+1}}) \\
        &= \bigsqcup_{(x_0, x_1)\in S}\dots \bigsqcup_{(x_{n-1}, x_n) \in \beta\kappa(e_{x_{n-2}x_{n-1}})}\bigsqcup_{(x_n, x_{n+1})\in \beta\kappa(e_{x_{n-1}x_n})} \beta\kappa(e_{x_nx_{n+1}})\\
        &= (\beta\kappa)^{n+1}S.
    \end{align*}
\end{proof}



Now we show that the class of tuples $\Theta_n(S)$ is a generalization of Kaneta and Yoshinaga's {\it thin frames} (\cite{KY}), which they mainly consider in the ``4-cut-free situation''.

\begin{df}[\cite{KY} Definition 5.1]\label{df4cut}
    A tuple $(x_0, \dots, x_n) \in X^{n+1}_f$ is a {\it thin frame} if it satisfies the following.
    \begin{enumerate}
    \item $\lnot(x_{i-1}\leq x_{i} \leq x_{i+1})$ for all $1\leq i \leq n-1$.
    \item $x_{i}\leq a \leq  x_{i+1}$ implies that $a = x_{i}$ or $a = x_{i+1}$ for all $0\leq i \leq n-1$.
    \end{enumerate}
\end{df}

\begin{df}
    A tuple $(x_0, x_1, x_2, x_3) \in X^{4}_f$ is called a {\it 4-cut} if it satisfies the following.
    \begin{enumerate}
        \item $x_0\leq x_1\leq x_2, x_1 \leq x_2 \leq x_3$ and $x_1 \neq x_2$,
        \item $\lnot(x_0\leq x_1 \leq x_3)$ or equivalently $\lnot(x_0\leq x_2 \leq x_3)$.
    \end{enumerate}
    We also define $m_X := \inf \{\sum_{i=0}^{3}d(x_i, x_{i+1}) \mid (x_0, x_1, x_2, x_3) \text{ is a 4-cut} \}\in [0, \infty]$.
\end{df}

\begin{prop}\label{thintheta}
If $\sum_{i=0}^{n-1}d(x_i, x_{i+1})<m_X$, then a tuple $(x_0, \dots, x_n) \in X^{n+1}_f$ is a thin frame if and only if it is in $\Theta_n(\beta(X^2_{f+}))$.

\end{prop}
\begin{proof}
    It is obvious that any thin frame $(x_0, \dots, x_n) \in X^{n+1}$ is in $\Theta_n(\beta(X^2_{f+}))$. We show the other inclusion. Let $(x_0, \dots, x_n) \in\Theta_n(\beta(X^2_{f+}))$. Then we have $\lnot(x_{i-1}\leq x_{i} \leq x_{i+1})$ for all $1\leq i \leq n-1$. Suppose that we have $x_{i}\leq a \leq  x_{i+1}$ for some $a\neq x_i, x_{i+1} \in X$. When $1 \leq i \leq n-1$, we obtain $x_{i-1} \leq x_i \leq a$. Since $(x_{i-1}, x_i, a,  x_{i+1})$ is not a 4-cut by the assumption, we obtain that $x_{i-1}\leq x_{i} \leq x_{i+1}$ that is a contradiction. Hence we conclude that $a= x_i$ or $a= x_{i+1}$. When $i= 0$, namely $x_{0}\leq a \leq  x_{1}$, it should be $a= x_0$ or $a= x_{1}$ since we have $(x_0, x_1) \in \beta(X^2_{f+})$.
\end{proof}

\subsection{Description of magnitude homology via $\Theta_n$}
It is shown in Proposition 1.14 of \cite{I} and Appendix of \cite{HR} that the magnitude homology of digraphs preserves filtered colimits. Here, we show that the magnitude homology also preserves filtered colimits in $\Met^{\sf inc}$; the category of quasi metric spaces and isometric embeddings. We consider the magnitude homology $\MH^\ell_n$ as a functor $\Met^{\sf inc} \too \Ab$. Then we have the following.

\begin{prop}\label{prescol}
    The functor $\MH^\ell_n$ preserves filtered colimits.
\end{prop}
\begin{proof}
 Let $U\colon \Met^{\sf inc} \too \Set$ be the forgetful functor.  For a filtered colimit ${\sf colim}_\alpha X_\alpha$ in $\Met^{\sf inc}$, we have a natural map ${\sf colim}_\alpha U(X_\alpha) \too U({\sf colim}_\alpha X_\alpha)$. We can verify that it is an isomorphism. Since filtered colimits commute with finite limits, we obtain an isomorphism ${\sf colim}_\alpha (U(X_\alpha))^{n+1} \too (U({\sf colim}_\alpha X_\alpha))^{n+1}$ for all $n\geq0$. We can define the length  $\sum_{i=0}^{n-1}d(x_i, x_{i+1})$ of a tuple $(x_0, \dots, x_n) \in {\sf colim}_\alpha (U(X_\alpha))^{n+1}$ by defining $d(x_i, x_{i+1}) = d(x_i^\alpha, x_{i+1}^\alpha) $ for some $\alpha$ such that $x_i^\alpha, x_{i+1}^\alpha \in X_\alpha$. It is immediately verified that the above isomorphism preserves the length of tuples which is also defined on the right hand side. Together with the fact that the free abelian group functor preserves colimits, we obtain an isomorphism $\MC^\ell_\ast({\sf colim}_\alpha X_\alpha) \cong {\sf colim}_\alpha\MC^\ell_\ast(X_\alpha)$ for all $\ell \geq 0$. Now the statement follows from the colimit preservation of the homology of chain complexes.
\end{proof}
\begin{rem}
    Note that $U$ is represented by the one point metric space $\ast$, namely $U(-) \cong \Met^{\sf inc}(\ast, -)$. The key in the proof of \cref{prescol} is that the one point metric space is a compact object in $\Met^{\sf inc}$. As mentioned in \cite{HR}, this is not true when we consider the category of quasi metric spaces and Lipschitz maps.
\end{rem}
Let $X$ be a finite geodetic quasi metric space with $|X| = N$. Recall from Section \ref{mhdf} that we have a projective resolution
\begin{align}\label{mgresol}
    \dots \xlongrightarrow{\del_{n+1}} \Z X^{n+2}_f \xlongrightarrow{\del_{n}}  \Z X^{n+1}_f  \xlongrightarrow{\del_{n-1}} \dots \xlongrightarrow{\del_{1}} \Z X^2_f \xlongrightarrow{\del_{0}} \Z^N \too 0
\end{align}
of the trivial $\mr{X}$-module $\Z^N \cong \mr{X}/rX$. Now we consider a set
\[
\Theta'_n = \{(x_0, \dots, x_{n+1})\mid (x_0, \dots, x_{n}) \in \Theta_n(\beta(X^2_{f+})) \}\subset X^{n+2}_f,
\]
for $n\geq 1$ and $\Theta'_0 = X^2_f$. Let $\Z\Theta'_n$ be a $\mr{X}$-submodule of $\Z X^{n+2}_f$ for $n\geq 0$. It is immediately verified that $\Z\Theta'_\ast \too \Z^{N} \too 0$ is a subchain complex of (\ref{mgresol}).

\begin{prop}\label{thetap}
    The chain complex $\Z\Theta'_\ast \too \Z^{N} \too 0$ is isomorphic to the chain complex $P_\ast \too \Z^N \too 0$ defined in Proposition \ref{mprgeod}, hence a minimal projective resolution of $\Z^N$.
\end{prop}
\begin{proof}
Recall that $P_0 = \mr{X}$ and $P_n = \bigoplus_{(x_0, \dots, x_{n}) \in \Theta_n(\beta(X^2_{f+}))}e_{x_n}\mr{X}$ for $n\geq 1$ by Proposition \ref{mprgeod} and Lemma \ref{theta3}. Hence we can identify $\Z\Theta'_n$ with $P_n$ for all $n\geq 1$ by sending $(x_0, \dots, x_{n+1}) \in \Theta'_n$ to $e_{x_nx_{n+1}} \in e_{x_n}\mr{X}$ in the $(x_0, \dots, x_n)$-component of $\bigoplus_{(x_0, \dots, x_{n}) \in \Theta_n(\beta(X^2_{f+}))}e_{x_n}\mr{X}$. We also have $P_0 = \Z\Theta'_0 = \Z X^2_f$. Furthermore, it is obvious that this identification is extended to an isomorphism of $\mr{X}$-modules, and that the boundary operators coincide.

\end{proof}



\begin{prop}\label{nongrad}
    For a geodetic quasi metric space $X$, the homology $\wt{\MH}_n(X)$ of the chain complex $(\wt{\MC}_\ast(X), \wt{\del}_\ast)$ in Section \ref{mhdf} is a free module with a basis $\Theta_n(\beta(X^2_{f+}))$.
\end{prop}
\begin{proof}
We first prove it for finite geodetic quasi metric space. Then the general case follows from Proposition \ref{prescol} and that any geodetic quasi metric space can be expressed as a sequential colimit of finite geodetic quasi metric spaces in $\Met^{\sf inc}$ by Lemma \ref{subgeod}.

    By Proposition \ref{thetap}, we have two projective resolutions $\Z\Theta'_\ast \too \Z^{N}$ and $\Z X^{\ast+2} \too \Z^N$, where the former is a subchain complex of the latter. Note that this inclusion is a homotopy equivalence between two projective resolutions of the same module by the fundamental theorem of homological algebra. Hence, by applying $-\otimes_{\mr{X}}\Z^N$ to these chain complexes, we obtain a homomorphism  $\Z\Theta'_\ast\otimes_{\mr{X}}\Z^N \too \Z X^{\ast+2}_f\otimes_{\mr{X}}\Z^N$ that induces an isomorphism on homology. Also, this homomorphism is obviously an inclusion. Since $\Z\Theta'_\ast \too \Z^{N}$ is a minimal projective resolution, any element of $\Z\Theta'_\ast\otimes_{\mr{X}}\Z^N $ is a cycle by \cref{zerodiff}. Furthermore, The identification $\Z X^{\ast+2}\otimes_{\mr{X}}\Z^N = \wt{\MC}_n(X)$ is restricted to the identification $\Z\Theta'_\ast\otimes_{\mr{X}}\Z^N = \Z \Theta_n (\beta(X^2_{f+}))$.
\end{proof}

Recall from \cite{HW} that the chain complex $(\wt{\MC}_\ast(X), \wt{\del}_\ast)$ can be decomposed as $\wt{\MC}_\ast(X) = \bigoplus_{\ell \geq 0}\wt{\MC}_\ast^\ell(X)$, where $\wt{\MC}_n^\ell(X) = \Z \langle (x_0, \dots, x_n) \in X^{n+1}_f \mid \sum_{i=0}^{n-1}d(x_i, x_{i+1})=\ell)\rangle$. The homology of each direct summand is denoted as $\MH_\ast^\ell(X)$.

\begin{thm}\label{grad}
    Let $X$ be a geodetic quasi metric space. Then, for all $\ell>0, n>0$, the magnitude homology $\MH_n^\ell(X)$ is a free module with a basis $\Theta_n^{\ell}(\beta(X^2_{f+})) = \{(x_0, \dots, x_n) \in \Theta_n(\beta(X^2_{f+})) \mid \sum_{i=0}^{n-1}d(x_i, x_{i+1})=\ell \}$. Namely, the basis of $\MH_n^\ell(X)$  is the set of tuples $(x_0, \dots, x_n)$ satisfying the following.
    \begin{enumerate}
    \item $\sum_{i=0}^{n-1}d(x_i, x_{i+1})=\ell$.
    \item $\lnot(x_{i-1}\leq x_{i} \leq x_{i+1})$ for $1\leq i \leq n-1$.
    \item $x_{0}\leq a \leq  x_{1}$ implies that $a = x_0$ or $a = x_1$.
    \item $x_{i}\leq a \leq  x_{i+1}$ and $a \neq x_{i+1}$ implies that $x_{i-1}\leq x_{i} \leq a$ for $1\leq i \leq n-1$.
    \end{enumerate}
\end{thm}
\begin{proof}
    Since we have $\Theta_n^{\ell}(\beta(X^2_{f+})) \subset \wt{\MC}^\ell_n(X)$ and $\Theta_n(\beta(X^2_{f+})) = \bigsqcup_{\ell}\Theta_n^{\ell}(\beta(X^2_{f+}))$, the statement follows from Proposition \ref{nongrad}.
\end{proof}

\begin{rem}\label{gradrem}
    In \cite{AI}, magnitude homology is described as a $\Tor$ functor taken in the abelian category of graded modules over a graded ring. However, we consider the category of non-graded modules in this paper, which is sufficient. It is because the statement for non-graded version (\cref{nongrad}) implies the graded version (\cref{grad}) due to the homogeneous decomposition of $\Theta_n(\beta(X^2_{f+}))$ in the proof of \cref{grad}.
\end{rem}

\begin{rem}\label{KYrem}
By Proposition \ref{thintheta}, Proposition \ref{grad} is a generalization of the following result by Kaneta and Yoshinaga in the sense that we overcome the obstruction of 4-cuts.

\begin{prop}[\cite{KY}]
Let $X$ be a geodetic quasi metric space. For $0 < \ell < m_X$, we have $\MH^{\ell}_n(X) \cong \Z\langle n\text{-thin frames of length } \ell \rangle$
\end{prop}

\end{rem}

\subsection{Geodetic Menger convex  quasi metric spaces}

\begin{df}
    A quasi metric space $(X, d)$ is {\it Menger convex} if, for all $x, y \in X$ with $d(x, y) < \infty$, there exists $z \neq x, y \in X$ such that $x \leq z \leq y$.
\end{df}

From \cref{grad}, we have the following corollary that is a generalization of Theorem 7.2 of \cite{J} that restricts the statement to metric spaces with no 4-cuts.

\begin{cor}\label{menger}
    For a geodetic Menger convex quasi metric space $X$, we have $\MH^\ell_n(X)=0$ for all $\ell>0, n>0$.
\end{cor}

Namely, we again get rid of the difficulty due to 4-cuts.


\section{Criterion for diagonality}

In this section, we show a criterion for the diagonality of a geodetic finite quasi metric space.
The {\it diagonality} of a graph is an intriguing property in magnitude theory. There are many examples of diagonal graphs (\cite{HW}, \cite{Gu}, \cite{TY}), and it is closely related to vanishing of path homology (\cite{As}).

Recall from \cite{HW} that a  graph $G$ is called {\it diagonal} if we have $\MH^{\ell}_n(G) = 0$ for $\ell\neq n$. More generally, recall from \cite{BK} that a quasi metric space $X$ is {\it diagonal} if any cycle of $\MH^{\ast}_\ast(X)$ is a linear combination of {\it saturated} tuples in the following sense.
\begin{df}[\cite{BK} Section 1.1]
   For a quasi metric space $X$,  a tuple $(x_0, \dots, x_n) \in X^{n+1}_f$ is {\it saturated} if $x_i\leq a \leq x_{i+1}$ implies $a = x_i$ or $a = x_{i+1}$ for all $0\leq i \leq n-1$.
\end{df}

We technically use the following term.
\begin{df}
     A quasi metric space $X$ is $n$-{\it diagonal} if any cycle of  $\MH^{\ell}_n(X)$ is a linear combination of saturated tuples for all $\ell\geq 0$.
\end{df}
Note that a graph $G$ is $n$-diagonal if and only if $\MH^{\ell}_n(G) = 0$ for all $\ell \neq n$. Also, a tuple $(u, v)$ in a graph is saturated if and only if $d(u, v) = 1$.
\begin{thm}\label{dcritf}
    Let $X$ be a finite geodetic quasi metric space. Then the following are equivalent.
    \begin{enumerate}
        \item $X$ is diagonal.
        \item $X$ is 2-diagonal.
        \item There is no 4-cut in $X$.
    \end{enumerate}

\end{thm}

\begin{proof}

$(1)\Rightarrow (2) : $ Obvious.

$(2) \Rightarrow (3) : $ It follows from Lemmas \ref{4cut1} and \ref{4cut2}, that will be proved below.

    $(3)\Rightarrow (1) : $ By \cref{thintheta,grad}, $\MH^\ell_n(G)$ is spanned by thin frames.
\end{proof}

\begin{rem}
    The arguments below and \cref{dcritf} can be applied to the case that the geodetic quasi metric space $X$ is ``discrete'' in the following sense : for all $x, y \in X$ and $a\neq y \in I_{X}(x, y)$, there exists $b\neq a \in I_X(x, y)$ such that $a\prec b$ and $\neg(a \leq c \leq b)$ for all $c \neq a, b \in X$. For example, \cref{dcritf} can be applied to any infinite geodetic digraphs.
\end{rem}
\begin{rem}\label{infremark}
    Remarkably, \cref{dcritf} does not hold for general geodetic quasi metric spaces, and even for classical metric spaces. For example, Jubin (\cite{J} Example 6.6) constructed a geodetic Menger convex metric space with a 4-cut, which is diagonal since its magnitude homology is trivial by \cref{menger}. For another example, which is non-symmetric and admits infinite distance, consider an interval $([0, p], d)$ with $d(s, t) = \begin{cases}
        t-s &  s \leq t\\ \infty & s>t
    \end{cases}$, and construct a metric space $([0, 1]\bigsqcup [0, 3])/0\sim 0, 1\sim 3$. This is Menger convex and geodetic, hence diagonal, but has 4-cuts.
\end{rem}

In the rest of this section,
we show lemmas used in the proof of \cref{dcritf}.

\begin{lem}\label{metlem}
    For a quasi metric space $(\{x, y, z, w\}, d)$ with $d(x, y), d(z, w) < \infty$, we have
    \[
    \begin{cases}x\leq y \leq z\\ x \leq z\leq w\end{cases}\Leftrightarrow \begin{cases}x\leq y \leq w\\ y\leq z \leq w\end{cases}.
    \]

\end{lem}
\begin{proof}
    Suppose that we have $x\leq y \leq z$ and $x \leq z\leq w$. Then we have $d(x, w) \leq d(x, y) + d(y, w) \leq d(x, y) + d(y, z) + d(z, w) = d(x, z) + d(z, w) = d(x, w)$, hence $x \leq y \leq w$ and $y\leq z \leq w$ by  $d(x, y) < \infty$. Conversely, suppose that $x\leq y \leq w$ and $y\leq z \leq w$. Then we have $d(x, w) \leq d(x, z) + d(z, w) \leq d(x, y) + d(y, z) + d(z, w) = d(x, y) + d(y, w) = d(x, w)$, and we have $x\leq y \leq z$ and $x \leq z\leq w$ similarly to the above.
\end{proof}
\begin{lem}\label{4cutlem}
    For a quasi metric space $X$, let $(x, y, z, w) \in X^4_f$ be a 4-cut and $A = \{a \in X\mid x\leq y\leq a, z \leq a \leq w\}$. Then $a \in A$ and $z \leq b \leq a$ implies $b \in A$.
\end{lem}
\begin{proof}
    By \cref{metlem} , we have the following.
    \begin{align*}
    \begin{cases}
       y \leq z \leq w  \\ z \leq a \leq w
        \end{cases}
        \Rightarrow
        \begin{cases}
            y \leq z \leq a \\ y \leq a \leq w
        \end{cases} \\
        \begin{cases}
        z \leq b \leq a \\ z \leq a \leq w
        \end{cases}
         \Rightarrow
        \begin{cases}
            z \leq b \leq w \\ b \leq a \leq w
        \end{cases} \\
    \begin{cases}
            x \leq y \leq a \\ y \leq z \leq a
        \end{cases}
         \Rightarrow
        \begin{cases}
            x \leq y \leq z  \\ x \leq z \leq a
        \end{cases} \\
        \begin{cases}
             x \leq z \leq a \\ z \leq b \leq a
        \end{cases}
         \Rightarrow
        \begin{cases}
            x \leq z \leq b  \\ x \leq b \leq a
        \end{cases} \\
        \begin{cases}
       x \leq y \leq z \\ x \leq z \leq b
        \end{cases}
         \Rightarrow
        \begin{cases}
            x \leq y \leq b \\ y \leq z \leq b
        \end{cases} \\
    \end{align*}
Note here that we have $d(a, w), d(z, b), d(z, a), d(b, a) < \infty$ by the assumptions. Hence we obtain $b \in A$.
\end{proof}

\begin{lem}\label{4cut1}
Let $X$ be a finite geodetic  quasi metric space. If $X$ is 2-diagonal, then there is no 4-cut $(x, y, z, w)\in X^4_f$ with $(x, y)$ being saturated.
\end{lem}
\begin{proof}
    Suppose that we have a 4-cut $(x, y, z, w) \in X^4_f$ with $(x, y)$ being saturated. By the geodeticity assumption, we have a totally ordered set $I_X(z, w) = \{z = t_0 \prec t_1\prec \dots \prec t_n \prec t_{n+1} = w\}$. By \cref{4cutlem}, there is some $0\leq k\leq n$ such that $\{a \in X\mid x\leq y\leq a, z \leq a \leq w\} = \{t_0, \dots, t_k\}$. Then we have $\begin{cases} y \leq z \leq t_{k+1} \\ y \leq t_{k+1}\leq w \end{cases}$ by $\begin{cases} y \leq z \leq w \\ z \leq t_{k+1} \leq w \end{cases}$ and \cref{metlem}  , which implies that $(x, y, z, t_{k+1})$ is a 4-cut.

    Now we show that $(x, y, t_{k+1}) \in \Theta_2(\beta(X^2_{f+}))$, which contradicts the 2-diagonality by \cref{grad} and $y\leq z \leq t_{k+1}, y\neq z, z \neq t_{k+1}$ . Since $(x, y)$ is saturated, $x \leq a \leq y$ implies $a = x$ or $a = y$. Suppose that $y \leq a \leq t_{k+1}$ and $a \neq t_{k+1}$. By $\begin{cases} y \leq a \leq t_{k+1} \\ y \leq t_{k+1}\leq w \end{cases}$ and \cref{metlem}  , we have $\begin{cases} y \leq a \leq w \\ a \leq t_{k+1} \leq w \end{cases}$. By the geodeticity assumption, we have $ y\leq a \leq z$ or $y \leq z \leq a$. When $ y\leq a \leq z$, we have $\begin{cases} x \leq y \leq a \\ x \leq a \leq z \end{cases}$ by $\begin{cases} x \leq y \leq z  \\ y \leq a \leq z\end{cases}$ and \cref{metlem}  , which implies $(x, y, t_{k+1}) \in \Theta_2(\beta(X^2_{f+}))$. When $y \leq z \leq a$, then we have $z \leq a \leq w$ and hence $a \in \{z, t_1, \dots, t_n\}$. By $y \leq a \leq t_{k+1}$ and $a \neq t_{k+1}$, we obtain that $a \in \{z, t_1, \dots, t_k\}$, which implies that $x \leq y \leq a$. Hence we have $(x, y, t_{k+1}) \in \Theta_2(\beta(X^2_{f+}))$.
\end{proof}

\begin{lem}\label{4cut2}
Let $X$ be a finite geodetic quasi metric space. If there is a 4-cut $(x, y , z, w) \in X^4_f$, then there is also a 4-cut $(x', y', z, w) \in X^4_f$ with $(x', y')$ being saturated.
\end{lem}

\begin{proof}
    By the geodeticity assumption, we have a totally ordered set $I_X(x, y) = \{x=s_0 \prec  s_1 \prec  \dots \prec s_m \prec s_{m+1} = y\}$. Since $y \leq z \leq w$ and $\neg(x \leq z \leq w)$, there should be some $1 \leq k \leq m+1$ such that $s_k \leq z \leq w$ and $\neg(s_{k-1} \leq z \leq w)$. Then we also have $s_{k-1} \leq s_{k} \leq z$, and hence $(s_{k-1}, s_k, z, w) \in X^4_f$ is a 4-cut with $(s_{k-1}, s_k)$ being saturated.
\end{proof}

\section{Magnitude homology of Moore graphs}
In this and next sections, we consider graphs as metric spaces by the shortest path metric. In this section, we determine the magnitude homology of all Moore graphs. A graph with degree $D$ and diameter $m$ is a  {\it Moore graph} if the number of vertices is maximum among such graphs. It has $1 + D\sum_{i=0}^{m-1}(D-1)^i$ vertices. Equivalently, a graph with diameter $m$ is a Moore graph if  its  girth is $2m+1$. Hoffman-Singleton's theorem (\cite{Brouwer} 6.7.1) states that Moore graphs are either complete graphs, odd cycle graphs ($D=2$), Petersen graph ($D=3$), Hoffman-Singleton graph ($D=7$) or a `missing Moore graph' ($D=57$) whose existence has not been shown. 

Note that the magnitude homology of a cycle graph is conjectured by Hepworth and Willerton in \cite{HW} and computed by Gu in \cite{Gu}. Here we give an alternative computation for this case by using minimal projective resolutions.

\begin{thm}\label{compute}
    Let $G$ be a Moore graph with degree $D$ and diameter $m > 1$. We denote the number of vertices by $N$. Then the magnitude homology $\MH^\ell_n(G)$ is a free $\Z$-module whose rank $R(n, \ell)$ is determined by the recurrence formula
    \[
    R(n, \ell) = R(n-1, \ell-1) + D(D-1)^mR(n-2, \ell - m-1),
    \]
    with the initial condition $R(0, 0) = N, R(1, 1) = ND$.
    Explicitly, we have $\MH^\ell_n(G)=0$ except for the case that \((n,\ell) = (2i+j, (m+1)i+j)\) for some \(i,j\ge 0\). In this case, we have
\begin{equation}\label{combi}
R(2i+j, (m+1)i+j) =  N\left(D(D-1)^m\right)^i\left({i+j-1 \choose i-1} + D {i+j-1 \choose i}\right),
\end{equation}
where \(s \choose t\) is the binomial coefficient and
  we define \({-1 \choose -1} = 1\) and \({s \choose t} = 0\) if \(s < t\) or \(t < 0 \leq s\).
\end{thm}
\begin{rem}\label{rem:moore_cycle}
 From the proof of \cref{compute}, we have explicit basis of  $\MH^{(m+1)i+j}_{2i+j}(G)$ which are tuples $(x_0, \dots, x_{2i+j})$ described as follows:
\begin{enumerate}
    \item $d(x_0, x_1) = 1, d(x_k, x_{k+1}) \in \{1, m\}$ for $1 \leq k \leq 2i+j-1$, and $k$ with $d(x_k, x_{k+1}) = m$ appears exactly $i$ times,
    \item $d(x_{k-1}, x_k) = d(x_k, x_{k+1}) = 1$ implies $x_{k-1} = x_{k+1}$,
    \item $d(x_k, x_{k+1}) = m$ implies $d(x_{k-1}, x_k) = 1 $ and $\neg(x_k \leq x_{k-1}\leq x_{k+1})$,
\end{enumerate}
Namely the cycles are tuples of an arrange of ``back-and-forth'' and ``distance $m$''. Such tuples are classified into two cases by the distance of the last two vertices. The number of tuples of each class is counted as follows  :
\begin{itemize}
    \item {\bf The case $d(x_{2i+j-1}, x_{2i+j}) = 1$} : We have $N$ choices of $x_0$ and $D$ choices of $x_1$ for a fixed $x_0$. In the tuple $(x_1, \dots, x_{2i+j})$,  there are $i$ triples $(x_{k-1}, x_k, x_{k+1})$ such that $d(x_{k-1}, x_{k}) = m, d(x_{k}, x_{k+1}) = 1 $, and all remained consecutive vertices are distance $1$ apart. The choice of distances ($1$ or $m$) of consecutive vertices in such a tuple is ${i+j-1 \choose i}$. Moreover, for such a triple $(x_{k-1}, x_k, x_{k+1})$, we have $(D-1)^m$ choices of $x_k$ for a fixed $x_{k-1}$ since we have $\neg(x_{k-2} \leq x_{k-1}\leq x_{k})$. Also, we have $D$ choices of $x_{k+1}$ for a fixed $x_{k}$. In total, we have $ND\left(D(D-1)^m\right)^i{i+j-1 \choose i}$ tuples which are linearly independent cycles of $\MH^{(m+1)i+j}_{2i+j}(G)$.
    \item {\bf The case $d(x_{2i+j-1}, x_{2i+j}) = m$} : We have $N$ choices of $x_0$ and $D$ choices of $x_1$ for a fixed $x_0$. In the tuple $(x_1, \dots, x_{2i+j-1})$, there are $(i-1)$ triples $(x_{k-1}, x_k, x_{k+1})$ such that $d(x_{k-1}, x_{k}) = m, d(x_{k}, x_{k+1}) = 1 $, and all remained consecutive vertices are distance $1$ apart. Similarly to the above, we have $(D(D-1)^m)^{i-1}{i+j-1 \choose i-1}$ choices of such tuples. Finally, we have $(D-1)^m$ choices of $x_{2i+j}$ for a fixed $x_{2i+j-1}$. In total, we have $N\left(D(D-1)^m\right)^i{i+j-1 \choose i-1}$ tuples which are linearly independent cycles of $\MH^{(m+1)i+j}_{2i+j}(G)$.
\end{itemize}
Note that the total number of the above tuples coincides with \cref{combi}.
\end{rem}

For the proof of \cref{compute}, we use the following lemma which follows immediately from the definition of $\Theta^{\ell}_n$.

\begin{lem}\label{rec}
    Let $X$ be a quasi metric space and $e_{xx'} \in X^2_{f+}$. Then we have
    \[
    \Theta^{\ell}_n(\{e_{xx'}\}) = \{e_{xx'}\}\times \Theta_{n-1}^{\ell - d(x, x')}(\beta\kappa(e_{xx'})) \subset X^{n+1}_{f+}.
    \]\qed

\end{lem}

 In the following, we denote $ \Theta^{\ell}_n(\{e_{xx'}\}) $ simply by $ \Theta^{\ell}_n(e_{xx'}) $. Note also that we have $\Theta_n^\ell(S) = \bigsqcup_{e \in S}\Theta_n^\ell(e)$ for all $S \subset X^2_{f+}$.

\begin{proof}[Proof of \cref{compute}]
Note that the condition $g = 2m+1$, where $g$ is the girth of $G$, implies that $G$ is geodetic.  Hence the freeness of the magnitude homology follows from \cref{nongrad}. From the definition, we have
\[
R(n, \ell) = |\Theta^{\ell}_n(\beta(X^2_+))| = \sum_{d(i, j)  =1}|\Theta^{\ell}_n(e_{ij})|.
\]
Also, we have $|\Theta^{\ell}_n(e_{ij})| = |\Theta^{\ell-1}_{n-1}(\beta\kappa(e_{ij}))|$ by \cref{rec}. Now the set $\kappa(e_{ij})$ with $d(i, j)  =1$ consists of pairs $e_{ji}$ and $e_{jk}$ with $d(j, k) = m$ by $g = 2m+1$. Further, the set $\beta\kappa(e_{ij})$ consists of pairs $e_{ji}$ and $e_{jk}$ with $d(j, k) = m, \neg(j\leq i \leq k)$. Let $T_{ij}$ be the set $\{ k \in V(G)\mid d(j, k) = m, \neg(j\leq i \leq k)\}$. Then, if $d(i, j)  =1$, we have
    \[
    |\Theta_n^\ell(e_{ij})| = |\Theta_{n-1}^{\ell-1}(e_{ji})| + \sum_{k\in T_{ij}}|\Theta_{n-1}^{\ell-1}(e_{jk})|.
    \]
    Also, any pair $e_{kp}$ belongs to the set $\kappa(e_{jk})$ with $k \in T_{ij}$. Further, the set $\beta\kappa(e_{jk})$ consists of pairs $e_{kp}$ with $d(k, p) = 1$. Hence we have
  \begin{align*}
  |\Theta_n^\ell(e_{ij})| &= |\Theta_{n-1}^{\ell-1}(e_{ji})| + \sum_{k\in T_{ij}}|\Theta_{n-1}^{\ell-1}(e_{jk})| \\
  &= |\Theta_{n-1}^{\ell-1}(e_{ji})| + \sum_{k\in T_{ij}}|\Theta_{n-2}^{\ell-m-1}(\beta\kappa(e_{jk}))| \\
   &= |\Theta_{n-1}^{\ell-1}(e_{ji})| + \sum_{k\in T_{ij}}\sum_{d(k, p) = 1}|\Theta_{n-2}^{\ell-m-1}(e_{kp})|,
  \end{align*}
  again by \cref{rec}. Now we  take $\sum_{d(i, j) = 1}$ and obtain
  \[
  R(n, \ell) = R(n-1, \ell-1) + \sum_{d(i, j) = 1}\sum_{k\in T_{ij}}\sum_{d(k, p) = 1}|\Theta_{n-2}^{\ell-m-1}(e_{kp})|.
  \]
  Note that, for all vertex $k \in V(G)$, the number of vertices $j$ satisfying $d(j, k) = m$ is $D(D-1)^{m-1}$. For such $j$, the number of vertices $i$ such that $k \in T_{ij}$ and $d(i, j) = 1$ is exactly $(D-1)$. Hence we conclude that $R(n, \ell) = R(n-1, \ell-1) + D(D-1)^mR(n-2, \ell - m - 1)$. The explicit formula immediately follows from \cref{lem:recurrence}.
\end{proof}


\begin{lem}
  \label{lem:recurrence}
  Let \(a,b,m,\lambda\in\Z\) and assume \(m\neq 1\).
  Then the condition
  \begin{equation*}
    R(n,\ell) =
    \begin{cases}
      a & \text{if \((n,\ell) = (0,0)\),} \\
      b & \text{if \((n,\ell) = (1,1)\),} \\
      0 & \text{if \(n < 0\) or \(\ell < 0\)}, \\
      R(n-1, \ell-1) + \lambda R(n-2, \ell-m-1) & \text{otherwise.}
    \end{cases}
  \end{equation*}
  uniquely determines the map \(R\colon \Z\times\Z\to\Z\).
  Moreover \(R(n,\ell)\) is given by the formula
  \begin{equation*}
    R(n,\ell) =
    \begin{cases}
      \lambda^i\left(a{i+j-1 \choose i-1} + b {i+j-1 \choose i}\right)
      & \text{if \((n,\ell) = (2i+j, (m+1)i+j)\) for some \(i,j\ge 0\)} \\
      0 & \text{otherwise,}
    \end{cases}
  \end{equation*}
  where \(s \choose t\) is the binomial coefficient and
  we define \({-1 \choose -1} = 1\) and \({s \choose t} = 0\) if \(s < t\) or \(t<0\leq s\).
\end{lem}
\begin{proof}
 Note that $R(n, \ell)$ is non-zero only if $(n, \ell)$ can be written as $j(1, 1) + i(2, m+1) = (2i+j, (m+1)i+j)$ for some $i, j \geq 0$. For \(i,j\in\Z\), we denote \(R'(i,j) := \lambda^{-i}R(2i+j, (m+1)i+j)\).
  By the definition of \(R(n,\ell)\), we have
  \(R'(0,0) = a\), \(R'(0,1) = b\) and
  \(R'(i,j) = R'(i,j-1) + R'(i-1,j)\) for \(i,j\ge 0\) with \((i,j) \neq (0,0), (0,1)\).
  Since the map \(\Z^2\to\Z^2;\ (i,j) \mapsto (2i+j, (m+1)i+j)\)
  is injective by the assumption \(m\neq 1\),
  we have \(R'(i,j) = 0\) when \(i<0\) or \(j<0\). Hence $R'(i, j)$ is the number of shortest paths from $(0, 0)$ to $(i, j)$ counted with weights; a path going through $(1,0)$ has weight $R'(1, 0) = a$ and a path going through $(0, 1)$ has weight $R'(0, 1) = b$.
  Namely we have \(R'(i,j) = a{(i-1)+j \choose i-1} + b{i+(j-1) \choose i}\),
  which completes the proof.
\end{proof}

Applying \cref{compute}, we give concrete computations for Moore graphs.
Note that explicit cycles can be obtained by \cref{rem:moore_cycle}.

\begin{eg}\label{egodd}
    Let $C_{2m+1}$ be the cycle graph with $2m+1$ vertices. Its degree and diameter are $2$ and $m$ respectively. The magnitude homology $\MH^\ell_n(C_{2m+1})$ is a free $\Z$-module of rank $R(n, \ell)$ determined by the recurrence formula
    \[
    R(n, \ell) = R(n-1, \ell-1) + 2R(n-2, \ell - m-1),
    \]
    with the initial condition $R(0, 0) = 2m+1, R(1, 1) = 4m+2$. We have $\MH^\ell_n(C_{2m+1}) = 0$ except for
   \[
   \rank \MH^{(m+1)i+j}_{2i+j}(C_{2m+1}) = (2m+1)2^i\left({i+j-1 \choose i-1} + 2{i+j-1 \choose i}\right),
   \]
with $i, j\geq 0$.
    \qed
\end{eg}

\begin{eg}\label{egpeter}
    Let $P$ be the Petersen graph (\cref{P}, \cite[A.3.1]{HW}). It  has $10$ vertices, and its degree and diameter are $3$ and $2$ respectively. The magnitude homology $\MH^\ell_n(P)$ is a free $\Z$-module of rank $R(n, \ell)$ determined by the recurrence formula
    \[
    R(n, \ell) = R(n-1, \ell-1) + 12R(n-2, \ell - 3),
    \]
    with the initial condition $R(0, 0) = 10, R(1, 1) = 30$.  We have $\MH^\ell_n(P) = 0$ except for
   \[
   \rank \MH^{3i+j}_{2i+j}(P) = 10\cdot12^i\left({i+j-1 \choose i-1} + 3{i+j-1 \choose i}\right),
   \]
with $i, j\geq 0$.
     \begin{figure}[H]
\centering

\begin{tikzpicture}

    \foreach \x in {0,...,4}
    \filldraw[fill=black, draw=black] ({3*cos((\x+1/4)*2*pi/5 r)}, {3*sin((\x+1/4)*2*pi/5 r)}) circle (3pt) node[left] {} ;

    \foreach \x in {0,...,4}
    \foreach \y in {0, ..., 4}
    \filldraw[fill=black, draw=black] ({1.5*cos((\x+1/4)*2*pi/5 r)}, {1.5*sin((\x+1/4)*2*pi/5 r)}) circle (3pt) node[left] {} ;

\foreach \j in {0, ..., 4}
    \draw[thick]  ({3*cos((\j+1/4)*2*pi/5 r)}, {3*sin((\j+1/4)*2*pi/5 r)})--({1.5*cos((\j+1/4)*2*pi/5 r)}, {1.5*sin((\j+1/4)*2*pi/5 r)})  ;

    \foreach \j in {0, ..., 4}
    \draw[thick]   ({3*cos((\j+1/4)*2*pi/5 r)}, {3*sin((\j+1/4)*2*pi/5 r)})--({3*cos((\j+1+1/4)*2*pi/5 r)}, {3*sin((\j+1+1/4)*2*pi/5 r)})  ;

    \foreach \j in {0, ..., 4}
    \draw[thick]  ({1.5*cos((\j+1/4)*2*pi/5 r)}, {1.5*sin((\j+1/4)*2*pi/5 r)})--({1.5*cos((\j+2+1/4)*2*pi/5 r)}, {1.5*sin((\j+2+1/4)*2*pi/5 r)})  ;

\end{tikzpicture}
\caption{Petersen graph}
      \label{P}
\end{figure}
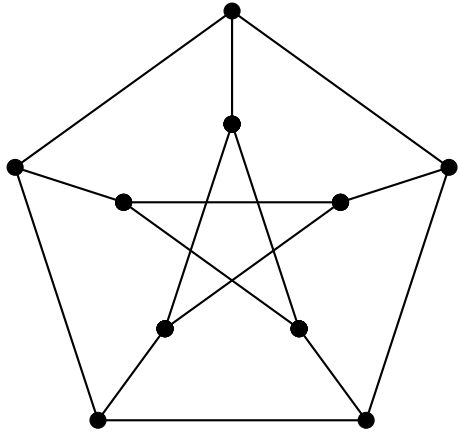\qed

\end{eg}

\begin{eg}\label{eghoffman}
    Let $HS$ be the Hoffman-Singleton graph (\cref{HS}). It has $50$ vertices and its degree and diameter are $7$ and $2$ respectively. The magnitude homology $\MH^\ell_n(HS)$ is a free $\Z$-module of rank $R(n, \ell)$ determined by the recurrence formula
    \[
    R(n, \ell) = R(n-1, \ell-1) + 252R(n-2, \ell - 3),
    \]
    with the initial condition $R(0, 0) = 50, R(1, 1) = 350$. We have $\MH^\ell_n(HS) = 0$ except for
   \[
   \rank \MH^{3i+j}_{2i+j}(HS) = 50\cdot252^i\left({i+j-1 \choose i-1} + 7{i+j-1 \choose i}\right),
   \]
with $i, j\geq 0$.

    \begin{figure}[H]
    \centering
     \begin{tikzpicture}

    \foreach \x in {0,...,4}
    \foreach \y in {0, ..., 4}
    \filldraw[fill=black, draw=black] ({4*cos((\x+\y/5)*2*pi/5 r)}, {4*sin((\x+\y/5)*2*pi/5 r)}) circle (3pt) node[left] {} ;

    \foreach \x in {0,...,4}
    \foreach \y in {0, ..., 4}
    \filldraw[fill=black, draw=black] ({2*cos((\x+\y/5)*2*pi/5 r)}, {2*sin((\x+\y/5)*2*pi/5 r)}) circle (3pt) node[left] {} ;

\foreach \j in {0, ..., 4}
    \foreach \h in {0,...,4}
    \foreach \i in {0, ..., 4}
    \draw  ({4*cos((\j-\h/5)*2*pi/5 r)}, {4*sin((\j-\h/5)*2*pi/5 r)})--({2*cos((\h*\i+\j+\i/5)*2*pi/5 r)}, {2*sin((\h*\i+\j+\i/5)*2*pi/5 r)})  ;

    \foreach \j in {0, ..., 4}
    \foreach \h in {0,...,4}
    \draw  ({4*cos((\j+\h/5)*2*pi/5 r)}, {4*sin((\j+\h/5)*2*pi/5 r)})--({4*cos((\j+1+\h/5)*2*pi/5 r)}, {4*sin((\j+1+\h/5)*2*pi/5 r)})  ;

    \foreach \j in {0, ..., 4}
    \foreach \h in {0,...,4}
    \draw ({2*cos((\j+\h/5)*2*pi/5 r)}, {2*sin((\j+\h/5)*2*pi/5 r)})--({2*cos((\j+2+\h/5)*2*pi/5 r)}, {2*sin((\j+2+\h/5)*2*pi/5 r)})  ;
\end{tikzpicture}
 \caption{Hoffman-Singleton graph}
      \label{HS}
\end{figure}
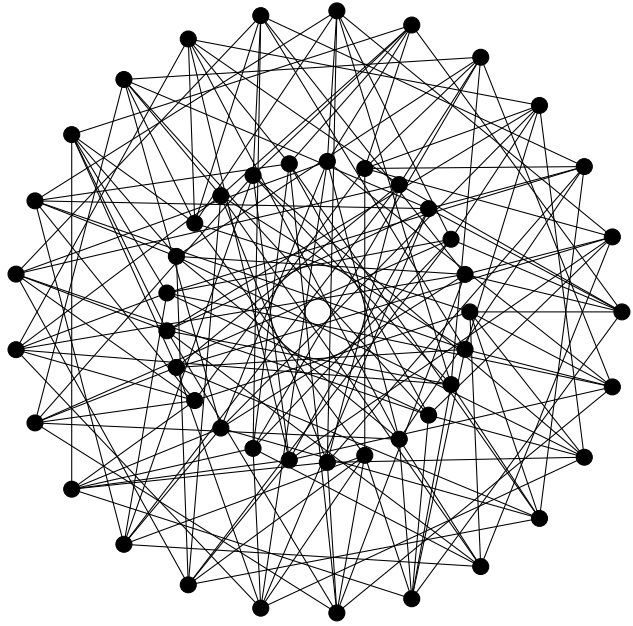
\qed
\end{eg}

\begin{eg}\label{egmoore}
    Suppose that we have a Moore graph $X$ with $D = 57$ (a missing Moore graph). It is known that if such a graph exists, its diameter is $2$ and it has $3250$ vertices. The magnitude homology $\MH^\ell_n(X)$ is a free $\Z$-module of rank $R(n, \ell)$ determined by the recurrence formula
    \[
    R(n, \ell) = R(n-1, \ell-1) + 178752R(n-2, \ell - 3),
    \]
    with the initial condition $R(0, 0) = 3250, R(1, 1) = 185250$. We have $\MH^\ell_n(X) = 0$ except for
     \[\rank \MH^{3i+j}_{2i+j}(X) = 3250\cdot178752^i\left({i+j-1 \choose i-1} + 57{i+j-1 \choose i}\right),
     \]
with $i, j\geq 0$.

    \qed
\end{eg}

\begin{rem}
    Since the odd cycle graphs, Petersen graph, and Hoffman-Singleton graph are vertex-transitive (even distance-transitive), we can compute their magnitude as
    \[
    {\rm Mag}(G) = |V(G)|/\sum_{y\in V(G)}q^{d(x, y)}
    \]
    for all $x \in V(G)$ by Speyer's lemma (\cite{L}). Although it is known that a missing Moore graph is not vertex transitive if it exists (\cite{Da}), we have a similar formula. More generally, we have the following formula for distance regular graphs containing Moore graphs. Here, a graph is {\it distance regular} if, for all $\ell \geq 0$, each two vertices $x, y$ have the same number of vertices  that is distance $\ell$ apart from $x$ and $y$ respectively.
    \begin{prop}
        For all finite distance regular graph $G$, we have
        \[
    {\rm Mag}(G) =  |V(G)|/\sum_{y\in V(G)}q^{d(x, y)}
    \] for all $x \in V(G)$.
    \end{prop}
    \begin{proof}
        Let $D_\ell$ be the number of vertices that is distance $\ell$ apart from $x \in V(G)$. Since $G$ is distance regular, $D_\ell$ does not depend on the choice of $x$. Also let $N$ be the number of vertices of $G$, and we fix a numbering $1, \dots, N$ on the vertices. Suppose that we have a magnitude weightings $w_1, \dots, w_N$, where $w_i$ is a weighting on the vertex $i$. Then, for all $j \in V(G)$, we have
        \[
        \sum_{\ell=0}^kq^{\ell}\sum_{d(j, i) = \ell}w_i = 1,
        \]
        where $k$ denotes the diameter of $G$. By taking $\sum_{j \in V(G)}$, we obtain
        \begin{align*}
            N = \sum_{j \in V(G)}1 &= \sum_{j \in V(G)}\sum_{\ell=0}^kq^{\ell}\sum_{d(j, i) = \ell}w_i \\
            &= \sum_{i=1}^N\sum_{\ell=0}^kD_\ell q^\ell w_i.
        \end{align*}
        Hence we obtain that
        \[
        \sum_{i=1}^{N}w_i = \frac{N}{\sum_{\ell=0}^kD_\ell q^\ell},
        \]
        which shows the statement.
    \end{proof}
\end{rem}

Hence the magnitude of a missing Moore graph is $3250/(1 + 57q + 3192q^2)$.

\section{Magnitude homology of even cycle graphs}
\newcommand{\possum}{a}
\newcommand{\negsum}{b}
\newcommand{\la}{\lmult{\possum}}
\newcommand{\lb}{\lmult{\negsum}}
As well as the odd case, the magnitude homology of even cycle graphs has been already studied in \cite{HW} and  \cite{Gu}. Here we give an alternative and more detailed computation, namely not only we determine the rank of the magnitude homology but also give explicit cycles of them.  Let \(N = 2m \ge 6\) be a positive even integer.
Since the graph \(C_N\) is not geodetic in this case,
we need to directly construct a (minimal) projective resolution of \(\Z^N\) over \(\mr{C_N}\).

\subsection{Statement}\label{statement : even}
In the following, we denote by $\Z^2$ the lattice in $\R^2$ spanned by $(1, 0)$ and $(0, 1)$. We also consider it as a graph with vertices $\{(p, q) \in \Z^2\}$ and edges $\{\{(p, q), (p', q')\} \mid |p-p'|+|q-q'| = 1\}$. For $p, q \geq 0$, let $S(p, q)$ be the set of $(p, q)$-shuffles, namely the set of all shortest paths from $(0, 0)$ to $(p, q)$ in $\Z^2$. For each vertex $x \in C_N$ and each $s \in S(p, q)$, we define a tuple $\varphi(x; s) \in (C_N)^{p+q+1}$ below (\cref{phidef}). Now we define a chain $\theta_{pq}(x) \in \MC_n(C_N)$ for each $p, q \geq0$ and $x \in C_N$ by
\[
\theta_{pq}(x) = \sum_{s\in S(p, q)}(-1)^{\nu(s)}\varphi(x; s),
\]
where the signature $(-1)^{\nu(s)}$ is defined below (\cref{nudef}).
\begin{prop}\label{thm:even}
  The magnitude homology \(\MH_{n}(C_N)\) is a free \(\Z\)-module
  whose basis is given by the set \(\{\theta_{pq}(x)\}_{\substack{x\in C_N\\ \ p+q=n}}\) of cycles in \({\MC}_n(C_N)\).
\end{prop}

By unpacking the explicit definitions of $\varphi$ and $\nu$ given below, \cref{thm:even} is refined as follows (c.f.\ \cref{gradrem}).

\begin{thm}\label{thm:evenref}
  \begin{enumerate}
    \item \(\MH_n^l(C_N) \neq 0\) if and only if \((n,l) = (2i+j, mi+j)\) for some \(i,j\in\N\).
    \item For each \(i\ge 0\), the magnitude homology \(\MH_{2i}^{mi}(C_N)\) is a free \(\Z\)-module of rank \(N\).
      Its basis is given by the set \(\{\theta_{ii}(x)\}_{x\in C_N}\) of cycles in \({\MC}_{2i}^{mi}(C_N)\).
    \item For each \(i\ge 0\) and \(j>0\),
      the magnitude homology \(\MH_{2i+j}^{mi+j}(C_N)\) is a free \(\Z\)-module of rank \(2N\).
      Its basis is given by the set \(\{\theta_{i+j,i}(x), \theta_{i,i+j}(x)\}_{x\in C_N}\) of cycles in \({\MC}_{2i+j}^{mi+j}(C_N)\).
  \end{enumerate}
\end{thm}

\subsection{Definition of $\varphi$ and $\nu$}
 Let $\xi : \R \too \R$ be a map defined by
\[
\xi(t) = \begin{cases}1 & t\geq 0 \\ 0 & t < 0\end{cases}.
\]

\begin{df}
 Let  $(p, q) \in \Z^2$ and $N=2m$. We define
    \begin{align*}
    \lambda((1, 0), (p, q)) &= (-1)^{p+q+1}((m-2)\xi(q-p)+1), \\
    \mu^{((1, 0),(p, q))} &= (-1)^{(p+q+1)\xi(q-p)}.
    \end{align*}
    We also define $\lambda((0, 1), (p, q)) = -\lambda((1, 0),(q, p))$ and $\mu^{((0, 1),(p, q))} = (-1)^{(p+q)\xi(p-q)}$.
    Namely we have maps
 \begin{align*}
     \lambda&: \{(1,0),(0,1)\}\times\Z^2 \to \{\pm 1, \pm (m-1)\}, \\
     \mu&: \{(1,0),(0,1)\}\times\Z^2 \to \{\pm 1\}.
 \end{align*}
\end{df}
Note that, for each $v\in\{(1,0),(0,1)\}$, we have $\mu^{(v,(p, q))} = -1$
if and only if $ \lambda(v, (p, q)) = -(m-1)$.
Let $\gamma$ be the rotation of $C_N$, namely an automorphism of $C_N$ that sends one vertex to the adjacent one (we don't care about the orientation of the rotation).
We denote a $(p, q)$-shuffle $s \in S(p, q)$ as a shortest path $s = (s_0, \dots, s_{p+q})$,
where $s_i \in \Z^2, s_0 = (0, 0)$ and $s_{p+q} = (p, q)$.
\begin{df}\label{phidef}
For $x \in C_N$ and $s = (s_0, \dots, s_{p+q}) \in S(p, q)$, we define a tuple $\varphi(x,s) = (x_0,\ldots,x_{p+q}) \in (C_N)^{p+q+1}$ by $ x_0 = x$ and
\[
x_i = \gamma^{\lambda(s_i-s_{i-1} , s_i)}x_{i-1},
\]
for $1\le i\le p+q$.
\end{df}
\begin{df}\label{nudef}
Let  $p, q \geq 0$ and $n := p+q$. For $s = (s_0,\ldots,s_{p+q}) \in S(p, q)$, we define
\[
\nu(s) = n(n+1)/2 + |\{i \mid \lambda(s_i-s_{i-1},s_i) = -(m-1)\}|.
\]
\end{df}

The following description helps us to understand the above definition; for $s = (s_0,\ldots,s_{p+q}) \in S(p, q)$, we have
\begin{align*}
(-1)^{\nu(s)} &= \prod_{i=1}^{n}(-1)^i\mu^{(s_i-s_{i-1},s_i)} \\
&= (-1)^{n+\nu(s_0, \dots, s_{n-1})}\mu^{(s_n-s_{n-1},s_n)}.
\end{align*}

\subsection{Example of a cycle in $\MH^\ell_n(C_N)$}

\begin{eg}
  Let \(x\in C_N\).
  We describe \(\theta_{2,1}(x) \in \MC_3^{m+1}(C_N)\) explicitly
  and prove \(\partial\theta_{2,1}(x) = 0\) directly.
  Since
  \begin{align*}
  S(2,1) = \bigl\{&\bigl((0, 0), (1, 0),(2, 0),(2, 1)\bigr), \\
  &\bigl((0, 0), (1, 0),(1, 1),(2, 1)\bigr), \\
  &\bigl((0, 0), (0, 1),(1,1),(2, 1)\bigr)\bigr\},
  \end{align*}
  we have
  \begin{align*}
    \theta_{2,1}(x)
    &= \sum_{s \in S(2,1)}
      (-1)^{\nu(s)}\varphi(x, s) \\
    &= -(x, \gamma x, x, \gamma^{-m+1}x) + (x, \gamma x, \gamma^m x, \gamma^{m+1}x) - (x, \gamma^{-1}x, \gamma^{-m}x, \gamma^{-m+1} x).
  \end{align*}
  Note that \(\gamma^{-m} x = \gamma^{m}x \in C_N\).
  Here we have
  \begin{itemize}
    \item \(\partial(x, \gamma x, x, \gamma^{-m+1}x) = (x, \gamma x, \gamma^{-m+1} x)\)
    \item \(\partial(x, \gamma x, \gamma^m x, \gamma^{m+1} x) = -(x, \gamma^m x, \gamma^{m+1}x) + (x, \gamma x, \gamma^{m+1}x)\)
    \item \(\partial(x, \gamma^{-1}x, \gamma^{-m}x, \gamma^{-m+1}x) = -(x, \gamma^{-m}x, \gamma^{-m+1}x)\)
  \end{itemize}
  and thus \(\partial\theta_{2,1}(x) = 0\).
  This example illustrates that each term in \(\theta_{2,1}(x)\) is not a cycle,
  but the sum \(\theta_{2,1}(x)\) of these terms is a cycle.
\end{eg}

\subsection{Step 1 for the proof : Construction of a free resolution}
In this subsection, we construct a free resolution of \(\Z^N\) over \(\mr{C_N}\) explicitly.

Define \(\possum, \negsum \in \mr{C_N}\) by
\[
\possum := \sum_{i=0}^{N-1}e_{\gamma^{i}x, \gamma^{i+1}x},\  \negsum := \sum_{i=0}^{N-1}e_{\gamma^{-i}x,\gamma^{-i-1}x},
\]
where $x$ is a fixed vertex of $C_N$.
Note that $a$ and $b$ do not depend on the choice of $x$.
We consider left multiplication maps
\(\la, \lb\colon \mr{C_N}\to\mr{C_N}\),
which are morphisms of right \(\mr{C_N}\)-modules.
Recall that \(\Z^N\) is the ``trivial'' \(\mr{C_N}\)-module
via the augmentation map \(\varepsilon: \mr{C_N} \to \mr{C_N}/rC_N \cong \Z^N\).
These maps satisfy the following relations:

\begin{lem}
  \label{lem:mult_rel}
  Let \(\la^k\) and \(\lb^k\) denote the \(k\)-times iteration of \(\la\) and \(\lb\), respectively.
  \begin{enumerate}
    \item \label{item:mult_comp} \(\la\lb = \lb\la = 0\) and \(\la^m = \lb^m\)
    \item \label{item:mult_exact} \(\ker\la = \im\lb\) and \(\ker\lb = \im\la\)
    \item \label{item:mult_im} \(\im\la^m = \im\lb^m = \im\la^{m-1}\cap\im\lb = \im\la\cap\im\lb^{m-1} = \im\la \cap \im\lb\)
    \item \label{item:mult_homol0} \(\ker\varepsilon = \im\la + \im\lb\)
    \item \label{item:mult_homol+}
      \(\ker\la^{m-1}\cap\ker\lb^{m-1} = \la(\ker\lb^{m-1}) + \lb(\ker\la^{m-1})\)
  \end{enumerate}
\end{lem}
\begin{proof}
  Note that it is enough to show these relations as morphisms of \textit{abelian groups}.
  Consider the direct sum decomposition
  \(\mr{C_N} = \bigoplus_x (\mr{C_N})e_x\)
  as abelian groups.
  Moreover, the morphisms \(\la\) and \(\lb\) decompose to direct sums
  \(\la = \bigoplus_x\la^{(x)}\) and
  \(\lb = \bigoplus_x\lb^{(x)}\),
  where \(\la^{(x)}, \lb^{(x)}\colon (\mr{C_N})e_x\to(\mr{C_N})e_x\).
  Hence the relations are reduced to the similar relations
  for \(\la^{(x)}\) and \(\lb^{(x)}\),
  which can be proved by straightforward and easy computation.
\end{proof}

\newcommand{\dcpxgen}{v}
\begin{df}
  We define a double complex \((\dcpx, \partialh, \partialv)\) by
  \begin{align*}
    \dcpx_{pq} &=
    \begin{cases}
      \mr{C_N} & p,q\ge 0 \\
      0 & \text{otherwise},
    \end{cases}\\
    \partialh_{pq} &=
    \begin{cases}
     \mu^{((1, 0), (p, q))}\lmult{a^{h(p, q)}} & \text{\(p>0\) and \(q\ge 0\)} \\
      0 & \text{otherwise},
    \end{cases} \\
    \partialv_{pq} &=
    \begin{cases}
       \mu^{((0, 1), (p, q))}\lmult{a^{v(p, q)}} & \text{\(q>0\) and \(p\ge 0\)} \\
      0 & \text{otherwise,}
    \end{cases}
  \end{align*}
where $h(p, q) = \lambda((1, 0),(p, q)), v(p, q) = \lambda((0, 1),(p, q))$, and we formally set $a^{-1}= b$.
\end{df}

We can verify from  \cref{lem:mult_rel} \cref{item:mult_comp} that the above $D_{pq}$'s define a double complex. It is described as follows.

\begin{equation*}
  \xymatrix{
    \vdots\ar[d]^-{\la} & \vdots \ar[d]^-{\lb} & \vdots \ar[d]^-{\la} & \vdots \ar[d]^-{\lb} & \\
    \mr{C_N} \ar[d]^-{\lb} & \mr{C_N} \ar[l]^-{-\lb^{m-1}} \ar[d]^-{\la} & \mr{C_N} \ar[l]^-{\la^{m-1}} \ar[d]^-{\lb} & \mr{C_N} \ar[l]^-{-\lb^{m-1}} \ar[d]^-{\la^{m-1}} & \cdots \ar[l]^-{\la} \\
    \mr{C_N} \ar[d]^-{\la} & \mr{C_N} \ar[l]^-{\la^{m-1}} \ar[d]^-{\lb} & \mr{C_N} \ar[l]^-{-\lb^{m-1}} \ar[d]^-{\la^{m-1}} & \mr{C_N} \ar[l]^-{\la} \ar[d]^-{-\lb^{m-1}} & \cdots \ar[l]^-{\lb} \\
    \mr{C_N} \ar[d]^-{\lb} & \mr{C_N} \ar[l]^-{-\lb^{m-1}} \ar[d]^-{\la^{m-1}} & \mr{C_N} \ar[l]^-{\la} \ar[d]^-{-\lb^{m-1}} & \mr{C_N} \ar[l]^-{\lb} \ar[d]^-{\la^{m-1}} & \cdots \ar[l]^-{\la} \\
    \mr{C_N} & \mr{C_N} \ar[l]^-{\la} & \mr{C_N} \ar[l]^-{\lb} & \mr{C_N} \ar[l]^-{\la} & \cdots \ar[l]^-{\lb} \\
  }
\end{equation*}

\begin{prop}
  \label{thm:resol_C2m}
  The following sequence is a free resolution of \(\Z^N\) over \(\mr{C_N}\):

  \begin{equation*}
    \cdots\to
    \tot_2 \dcpx \xrightarrow{\partial}
    \tot_1 \dcpx \xrightarrow{\partial}
    \tot_0 \dcpx \xrightarrow{\varepsilon}
    \Z^N \to 0
  \end{equation*}

  Moreover, \(\partial\otimes\id = 0\) on the chain complex
  \(\tot_\ast D \otimes_{\mr{C_N}} \Z^N\).
\end{prop}
\newcommand{\homolk}[2]{H_{#1}^{\mathrm{K}}(#2)}
\begin{proof}
  By \cref{lem:mult_rel} \cref{item:mult_exact} and \cref{item:mult_im},
  we can apply \cref{prop:homolk_tot}, that will be proved later, to this double complex
  and hence \(H_{2k-1}(\tot\dcpx) = 0\) for all \(k\in\Z\).
  By \cref{lem:mult_rel} \cref{item:mult_homol0} and \cref{item:mult_homol+},
  we have
  \(H_0(\tot\dcpx) \cong \homolk{00}{\dcpx} \cong \Z^N\) and
  \(H_{2k}(\tot\dcpx) \cong \homolk{kk}{\dcpx} = 0\) for \(k\ne 0\), respectively.
  Hence the above sequence is exact.

  Since \(a,b\in \mathrm{rad}(\mr{C_N}) = \ker(\varepsilon\colon\mr{C_N}\twoheadrightarrow\Z^N)\),
  we have \(\partial\otimes\id = 0\) on \(\tot\dcpx \otimes_{\mr{C_N}}\Z^N\).
\end{proof}

Note that
the above resolution is a minimal projective resolution by the property
\(a,b\in \mathrm{rad}(\mr{C_N}) = \ker(\varepsilon\colon\mr{C_N}\twoheadrightarrow\Z^N)\).

\subsection{Step 2 for the proof : Comparison of complexes}\label{comparison}
\newcommand{\paths}[1]{P(#1)}
In this subsection, we give a proof of \cref{thm:even} by comparing two complexes;
the complex $\tot_\ast\dcpx$ obtained in the previous subsection and  \((\MResol_\ast(X), \partial_\ast)\), the projective resolution of the $\mr{X}$-module $\Z^N$ constructed in \cref{mhdf}. That is, \(\MResol_n(X) = \Z X^{n+2}_{f}\) for \(n\ge 0\) and
\(\MResol_n(X) = 0\) for \(n < 0\).

\begin{df}
  For \(p,q\ge 0\), let \(\dcpxgen_{pq}\)  be the unit \(1\in\mr{C_N}=\dcpx_{pq}\). We define a map \(f\colon \tot_n\dcpx = \bigoplus_{p+q=n}v_{pq}\mr{C_N} \to \MResol_n(C_N)\)
  of \(\mr{C_N}\)-modules by
  \[
  f(v_{pq}) = \sum_{x\in C_N}\sum_{s\in S(p,q)}(-1)^{\nu(s)}\wt\varphi(x,s),
  \]
  where we set $\wt\varphi(x,s) = (x_0, \dots, x_{n}, x_n)$ in association with $\varphi(x,s) = (x_0, \dots, x_{n})$.
\end{df}

The following lemma is a  key to prove that \(f\) is a chain map.

\begin{lem}
  \label{lem:even_diff}
  For all \((x, (s_0,\ldots,s_n)) \in C_N \times S(p, q)\) with $p + q = n$, we have the following. Here $\del_{n, i}$'s are the homomorphisms defined in \cref{mhdf}.
  \begin{enumerate}
    \item \(\partial_{n,0}(\wt\varphi(x,(s_0,\ldots,s_n))) = 0\)
    \item \label{item:even_diff_middle}
      For \(1\le i \le n-1\), let $s'_i$ be the reflection of $s_i$ through the line $\overline{s_{i-1}s_{i+1}}$. Then
\[
\partial_{n,i}(\wt\varphi(x,(s_0,\ldots,s_n))) = \begin{cases}
    0 & s_i = s'_i \\
    \partial_{n,i}(\wt\varphi(x,(s_0,\ldots, s_{i-1}, s'_i, s_{i+1}, \dots, s_n))) & s_i\neq s'_i
\end{cases}.
\]
When \(s_i\neq s'_i\),
 we also have
          \(\nu(s_0,\ldots,s_n) - \nu(s_0,\ldots, s_{i-1}, s'_i, s_{i+1}, \dots, s_n) = \pm 1\).
    \item \(\partial_{n,n}(\wt\varphi(x,(s_0,\ldots,s_n))) = \wt\varphi(x,(s_0,\ldots,s_{n-1}))\cdot a^{\lambda(s_n-s_{n-1},s_n)}\).
  \end{enumerate}
\end{lem}
\begin{proof}
  (1) and (3) are immediate.
  For (2), recall that
  \begin{equation*}
    (x_{i-1}, x_i, x_{i+1}) = (x_{i-1}, \gamma^{\lambda(s_i-s_{i-1},s_i)}x_{i-1}, \gamma^{\lambda(s_i-s_{i-1},s_i)+\lambda(s_{i+1}-s_{i},s_{i+1})}x_{i-1})
  \end{equation*}
  for \((x_0, \dots, x_{n}, x_n) := \wt\varphi(x,s)\).
  When \(s_i=s'_i\), we have \(s_i-s_{i-1} = s_{i+1}-s_i\) and hence \(\neg(x_{i-1}\le x_i\le x_{i+1})\),
  which implies \(\partial_{n,i}(\wt\varphi(x,(s_0,\ldots,s_n)))=0\).
  When \(s_i\neq s'_i\), we have \((s_i-s_{i-1},s_{i+1}-s_i) \in \{((1,0),(0,1)), ((0,1),(1,0))\}\).
  In each case, we have \(x_{i-1}\le x_i\le x_{i+1}\) and \(x'_{i-1}\le x'_i\le x'_{i+1}\),
  where \((x'_0, \ldots, x'_{n}, x'_n) := \wt\varphi(x,(s_0,\ldots, s_{i-1}, s'_i, s_{i+1}, \ldots, s_n))\).
  Since \(x_k=x'_k\) except for \(k=i\), we have
  \begin{align*}
    \partial_{n,i}(\wt\varphi(x,(s_0,\ldots,s_n)))
    &= (x_0,\ldots, \hat{x}_i,\ldots, x_n, x_n) \\
    &= (x'_0,\ldots, \hat{x'}_i,\ldots, x'_n, x'_n) \\
    &= \partial_{n,i}(\wt\varphi(x,(s_0,\ldots, s_{i-1}, s'_i, s_{i+1}, \dots, s_n))).
  \end{align*}
  The equation for \(\nu\) follows from the fact that
  exactly one of
  \(\lambda(s_i-s_{i-1}, s_i)\),
  \(\lambda(s_{i+1}-s_i, s_{i+1})\),
  \(\lambda(s'_i-s_{i-1}, s'_i)\) or
  \(\lambda(s_{i+1}-s'_i, s_{i+1})\)
  is equal to \(-(m-1)\).
\end{proof}

\begin{prop}\label{dffd}
    The homomorphism \(f\colon \tot_\ast\dcpx \to \MResol_\ast(C_N)\) is a chain map over \(\mr{C_N}\).
\end{prop}
\begin{proof}
    Let $n = p+q$. By \cref{lem:even_diff} (1), we have
    \begin{align*}
        \del_n f(v_{pq}) &= \sum_{x \in C_N}\sum_{s \in S(p, q)}(-1)^{\nu(s)}\del_n \wt\varphi(x,s) \\
        &= \sum_{x \in C_N}\sum_{s \in S(p, q)}\sum_{i=0}^n(-1)^{\nu(s)+i}\del_{n, i} \wt\varphi(x,s) \\
        &= \sum_{x \in C_N}\sum_{s \in S(p, q)}\sum_{i=1}^{n-1}(-1)^{\nu(s)+i}\del_{n, i} \wt\varphi(x,s) +
        \sum_{x \in C_N}\sum_{s \in S(p, q)}(-1)^{\nu(s)+n}\del_{n, n} \wt\varphi(x,s).
    \end{align*}
    For each $1 \leq i \leq n-1$ and $x \in C_N$, we have $\sum_{s \in S(p, q)}(-1)^{\nu(s)+i}\del_{n, i} \wt\varphi(x,s) = 0$ by \cref{lem:even_diff} (2). Also, by \cref{lem:even_diff} (3), we have
    \begin{align*}
     \sum_{x \in C_N}\sum_{s \in S(p, q)}(-1)^{\nu(s)+n}\del_{n, n} \wt\varphi(x,s) &= \sum_{x \in C_N}\sum_{t \in S(p-1, q)}(-1)^{\nu(t)}\mu^{((1, 0),(p, q))}\wt\varphi(x,t)a^{\lambda((1, 0), (p, q))}\\
     &\ \ \ + \sum_{x \in C_N}\sum_{t \in S(p, q-1)}(-1)^{\nu(t)}\mu^{((0, 1),(p, q))}\wt\varphi(x,t)a^{\lambda((0, 1), (p, q))} \\
     &= f(v_{p-1, q})\mu^{((1, 0),(p, q))}a^{\lambda((1, 0), (p, q))} \\
     &\ \ \ + f(v_{p, q-1})\mu^{((0, 1),(p, q))}a^{\lambda((0, 1), (p, q))}  \\
     &= f(\del_n v_{p, q}). \qedhere
    \end{align*}
\end{proof}

\begin{prop}
  \label{prop:MR_equiv}
  The homomorphism \(f\colon \tot_\ast\dcpx \to \MResol_\ast(C_N)\) is a chain homotopy equivalence over \(\mr{C_N}\).
\end{prop}
\begin{proof}
  By \cref{thm:resol_C2m,barresol},
  both \(\tot_\ast\dcpx\) and \(\MResol_\ast(C_N)\) are projective resolutions of \(\Z^N\) over \(\mr{C_N}\).
  Hence the map \(f\) is a chain homotopy equivalence.
\end{proof}

Now we are ready to prove \cref{thm:even}.
\begin{proof}[Proof of \cref{thm:even}]
  By \cref{prop:MR_equiv}, we have
  \begin{equation*}
    f_\ast\colon
    H_n(\tot\dcpx \otimes_{\mr{C_N}}\Z^N)
    \xrightarrow{\cong} H_n(\MResol(C_N)\otimes_{\mr{C_N}}\Z^N)
    = \MH_n(C_N).
  \end{equation*}
  Here \cref{thm:resol_C2m} implies
  \(H_n(\tot\dcpx \otimes_{\mr{C_N}}\Z^N) = \tot_n\dcpx \otimes_{\mr{C_N}}\Z^N\) and
  its \(\Z\)-basis is given by \(\{v_{pq}\otimes e_x\}_{\substack{x\in C_N\\ \ p+q=n}}\), 
  whose image under \(f_\ast\) is \(\{\theta_{pq}(x)\}_{\substack{x\in C_N\\ \ p+q=n}}\).
\end{proof}

\begin{rem}
  Note that the method given here
  can be also applied to the odd case,
  and the odd case is much simpler than the even case
  since \(\im\la\cap\im\lb = 0\)
  and a complicated proposition (\cref{prop:homolk_tot}) in homological algebra is not necessary.
  This coincides with the comment
  ``It turns out that magnitude homology of odd cycles has
  more complicated description but easier computation''
  given in \cite[Section 4.4]{Gu}.
\end{rem}

\subsection{Homological algebra}
\label{subsection:homol_alg}
In this subsection,
we show \cref{prop:homolk_tot}, which is used in the proof of \cref{thm:resol_C2m}.

Let \(\dcpx = (\{\dcpx_{pq}\}_{p,q\in\Z}, \{\partialh_{pq}\}, \{\partialv_{pq}\})\) be a double complex,
where
\(\partialh_{pq}\colon \dcpx_{pq}\to \dcpx_{p-1,q}\),
\(\partialv_{pq}\colon \dcpx_{pq}\to \dcpx_{p,q-1}\) and
\(\partialh\partialh = \partialv\partialv = \partialh\partialv + \partialv\partialh = 0\).
Denote the total chain complex by
\(\tot \dcpx = (\{\tot_n \dcpx\}_{n\in\Z}, \partial = \partialh + \partialv)\),
where \(\tot_n \dcpx = \bigoplus_{p+q=n}\dcpx_{pq}\).

\begin{df}
  For a double complex \(\dcpx\) and \(p, q\in \Z\), we define
  \begin{equation*}
    \homolk{pq}{\dcpx}
    := \left(\frac{\ker\partialh \cap \ker\partialv}{\partialh(\ker\partialv) + \partialv(\ker\partialh)}\right)_{pq}
    = \frac{\ker\partialh_{pq} \cap \ker\partialv_{pq}}{\partialh_{p+1,q}(\ker\partialv_{p+1,q}) + \partialv_{p,q+1}(\ker\partialh_{p,q+1})}.
  \end{equation*}
  Let \(\varphi_{pq}\colon \homolk{pq}{\dcpx} \to H_{p+q}(\tot \dcpx)\)
  be the morphism induced by the inclusion
  \(\ker\partialh\cap\ker\partialv\to\tot \dcpx\).
\end{df}

\begin{rem}
  The Bott-Chern (co)homology
  \begin{equation*}
    H_{pq}^{\mathrm{BC}}(\dcpx) = \left(\frac{\ker\partialh\cap\ker\partialv}{\im(\partialh\partialv)}\right)_{pq}
  \end{equation*}
  plays an important role in complex geometry (c.f.\ \cite{St}).
  Note that the above homology \(\homolk{pq}{\dcpx}\) is slightly different from \(H_{pq}^{\mathrm{BC}}(\dcpx)\)
  and they are related by the natural surjection \(H_{pq}^{\mathrm{BC}}(\dcpx) \to \homolk{pq}{\dcpx}\).
\end{rem}

Although the following two lemmas can be proved by standard arguments in homological algebra,
we give proofs here for convenience of the reader.

\begin{lem}
  \label{lem:varphi_inj_surj}
  Assume that
  \(\im\partialh_{p+1,q}\cap\im\partialv_{p,q+1} = \im(\partialh_{p+1,q}\partialv_{p+1,q+1})\)
  holds for all \(p,q \in \Z\).
  Then
  \begin{enumerate}
    \item \label{item:varphi_inj}
      the map \(\varphi_{pq}\colon \homolk{pq}{\dcpx} \to H_{p+q}(\tot \dcpx)\)
      is injective for all \(p,q \in \Z\), and
    \item \label{item:varphi_surj}
      \(\sum_{p+q=n}\varphi_{pq}\colon \bigoplus_{p+q=n}\homolk{pq}{\dcpx}\to H_n(\tot \dcpx)\)
      is surjective for all \(n \in \Z\).
  \end{enumerate}
\end{lem}
\begin{proof}
  To prove \cref{item:varphi_inj}, take any element \([x]\in\ker\varphi_{pq}\).
  Then there is a sequence
  \((y_{i,p+q-i+1})_i \in \tot_{p+q+1}\dcpx\)
  of \(y_{i,p+q-i+1} \in \dcpx_{i,p+q-i+1}\) satisfying
  \begin{equation*}
    \partialh y_{i+1,p+q-i} + \partialv y_{i,p+q-i+1} =
    \begin{cases}
      x & \text{if \(i=p\)} \\
      0 & \text{otherwise}
    \end{cases}
  \end{equation*}
  for all \(i\in\Z\).
  Since
  \begin{math}
    \partialv y_{p+1,q}
    = \partialh y_{p+2,q-1}
    \in \im\partialh \cap \im\partialv
    =\im(\partialh\partialv)
  \end{math},
  there is an element \(z\in \dcpx_{p+2,q}\) such that
  \(\partialv\partialh z = \partialv y_{p+1,q}\).
  Similarly we have \(w\in\dcpx_{p,q+2}\) such that
  \(\partialh\partialv w = \partialh y_{p,q+1}\).
  Hence we have
  \begin{equation*}
    x = \partialh(y_{p+1,q} - \partialh z) + \partialv(y_{p,q+1}-\partialv w)
    \in \partialh(\ker\partialv) + \partialv(\ker\partialh)
  \end{equation*}
  and thus \([x] = 0 \in \homolk{pq}{\dcpx}\).
  This proves \cref{item:varphi_inj}.

  To prove \cref{item:varphi_surj},
  take any element \([\mathbf{y}] \in H_n(\tot\dcpx)\).
  Here \(\mathbf{y}=(y_{p,n-p})_p\in\tot_n\dcpx\) is
  a sequence of elements \(y_{p,n-p}\in\dcpx_{p,n-p}\) such that
  \(y_{p,n-p} = 0\) except for finitely many \(p\in\Z\) and
  \(\partialh y_{p,n-p} + \partialv y_{p-1,n-p+1} = 0\) for all \(p\in\Z\).
  For each \(p\in\Z\), we have
  \begin{math}
    \partialh y_{p,n-p}
    = -\partialv y_{p-1,n-p+1}
    \in \im\partialh \cap \im\partialv
    = \im(\partialh\partialv)
  \end{math}
  and hence there is an element \(z_{p,n-p+1}\in\dcpx_{p,n-p+1}\) such that
  \(\partialh\partialv z_{p,n-p+1} = \partialh y_{p,n-p} = -\partialv y_{p-1,n-p+1}\).
  Note that we can take \((z_{p,n-p+1})_p\)
  so that \(z_{p,n-p+1} = 0\) except for finitely many \(p\in\Z\).
  Now we define
  \(x_{p,n-p} = y_{p,n-p} - \partialh z_{p+1,n-p} - \partialv z_{p,n-p+1} \in \dcpx_{p,n-p}\)
  for each \(p\in\Z\).
  Then a straightforward computation shows that
  \([x_{p,n-p}] \in \homolk{p,n-p}{\dcpx}\) and
  \(\sum_p\varphi_{p,n-p}[x_{p,n-p}] = [\mathbf{y}] \in H_n(\tot\dcpx)\),
  which completes the proof.
\end{proof}

\begin{lem}
  \label{lem:homolk_vanish}
  Fix \(s,t\in\Z\).
  Assume that \(\dcpx\) satisfies the assumption in \cref{lem:varphi_inj_surj} and either of the following:
  \begin{itemize}
    \item
      \(\ker\partialh_{st} = \im\partialh_{s+1,t}\) and
      \(\ker\partialh_{s+1,t-1} = \im\partialh_{s+2,t-1}\)
    \item
      \(\ker\partialv_{st} = \im\partialv_{s,t+1}\) and
      \(\ker\partialv_{s-1,t+1} = \im\partialv_{s-1,t+2}\)
  \end{itemize}
  Then we have \(\homolk{st}{\dcpx} = 0\).
\end{lem}
\begin{proof}
  Here we give a proof under the first assumption.
  Take any element \([x] \in \homolk{st}{\dcpx}\).
  Since \(x\in\ker\partialh_{st} = \im\partialh_{s+1,t}\),
  there is an element \(y\in\dcpx_{s+1,t}\) such that \(\partialh y = x\).
  Then we have
  \(\partialv y \in \ker\partialh_{s+1,t-1} = \im\partialh_{s+2,t-1}\) and hence
  \(\partialv y \in \im\partialh \cap \im \partialv = \im(\partialh\partialv)\).
  Thus there is an element \(z\in\dcpx_{s+2,t}\) such that \(\partialv\partialh z = \partialv y\).
  Now we have \(x = \partialh(y - \partialv z) \in \partialh(\ker\partialv)\),
  which completes the proof.
\end{proof}

Now we are ready to prove

\begin{prop}
  \label{prop:homolk_tot}
  Assume that the double complex \(\dcpx\) satisfies the assumption in \cref{lem:varphi_inj_surj} and
  \begin{itemize}
    \item \(\ker\partialh_{pq} = \im\partialh_{p+1,q}\) if \(p > q\), and
    \item \(\ker\partialv_{pq} = \im\partialv_{p,q+1}\) if \(q > p\)
  \end{itemize}
  for all \(p,q\in\Z\).
  Then, for all \(k \in \Z\), we have
  \begin{enumerate}
    \item \(H_{2k-1}(\tot \dcpx) = 0\) and
    \item \(\varphi_{kk}\colon \homolk{kk}{\dcpx} \to H_{2k}(\tot \dcpx)\) is an isomorphism.
  \end{enumerate}
\end{prop}
\begin{proof}
  By \cref{lem:homolk_vanish},
  we have \(\homolk{pq}{\dcpx} = 0\) for all \(p,q\in\Z\) with \(p\neq q\).
  Hence \cref{lem:varphi_inj_surj} completes the proof.
\end{proof}

\bibliographystyle{alpha_plainsort}
\bibliography{references}

\end{document}